\documentclass[a4paper,12pt]{article}

\usepackage{lineno}
\usepackage[T1]{fontenc}
\usepackage[utf8]{inputenc}
\usepackage{amsmath,amsthm,mathrsfs,amssymb,amsfonts,nicefrac,bbm}
\usepackage[english]{babel}
\usepackage[active]{srcltx}
\usepackage[top=3cm, bottom=3cm, left=2.8cm, right=2.8cm]{geometry}
\usepackage{tikz}
\usepackage{dsfont}\usepackage{graphicx}
\usetikzlibrary{patterns}
\usepackage{color}
\usepackage[colorlinks=true]{hyperref}

\usepackage{autonum} 
\usepackage{subfigure}

\newtheorem{theorem}{Theorem}[section]

\newtheorem{e-proposition}[theorem]{Proposition}

\makeatletter

\@addtoreset{equation}{section}
\@addtoreset{chapter}{part}

\newcommand{\RR}{\mathbb{R}}

\newcommand{\N}{\mathbb{N}}

\renewcommand{\leq}{\leqslant}
\renewcommand{\geq}{\geqslant}

\renewcommand{\tilde}{\widetilde}

\renewcommand{\t}{\tilde}

\newcommand{\blue}{}

\newcommand{\di}{\displaystyle}

\def\1{\mathbbm{1}}

\numberwithin{equation}{section}

\newtheorem{Lemma}{Lemma}[section]

\newtheorem{Th}[Lemma]{Theorem}
\newtheorem{Prop}[Lemma]{Proposition}
\newtheorem{Rk}[Lemma]{Remark}

\newcommand{\be}{\begin{equation}}
\newcommand{\ee}{\end{equation}}
\newcommand{\baa}{\begin{array}}
\newcommand{\eaa}{\end{array}}
\newcommand{\ba}{\begin{eqnarray}}
\newcommand{\ea}{\end{eqnarray}}

\newcommand{\vp}{\varphi}

\def\epsilon{\varepsilon}

\def\J{\mathcal{J}}
\def\II{\mathcal{I}}
\def\TT{\mathcal{T}}
\def\SST{\mathcal{S}}
\def\DD{\mathcal{D}}
\def\UU{\mathcal{U}}
\def\VV{\mathcal{V}}

\def\mc{\mathcal}

\def\Z{{\mathbb Z}}

\def\trait (#1) (#2) (#3){\vrule width #1pt height #2pt depth #3pt}
\def\fin{\hfill\trait (0.1) (5) (0) \trait (5) (0.1) (0) \kern-5pt \trait (5) (5) (-4.9) \trait (0.1) (5) (0)}
\def\todo#1 {\marginpar{\textcolor{red}{$\Rightarrow$#1}}}

\newcommand{\bee}{\begin{equation*}}
\newcommand{\eee}{\end{equation*}}
\newcommand{\bc}{\begin{cases}}
\newcommand{\ec}{\end{cases}}

\def\e{\varepsilon}

\def\cK{c_{\scalebox{0.5}{$K$}}}
\def\cSIR{c_{\scalebox{0.5}{$SIR$}}}
\def\cSIRT{c_{\scalebox{0.5}{$SIR$}}^{\scalebox{0.5}{$T$}}}
\def\wSIR{w_{\scalebox{0.5}{$SIR$}}}
\def\wSIRT{w_{\scalebox{0.5}{$SIR$}}^{\scalebox{0.5}{$T$}}}
\def\WSIRT{\omega_{\scalebox{0.5}{$SIR$}}^{\scalebox{0.5}{$T$}}}

\begin{document}
\title{\bf Biological invasions and epidemics with nonlocal diffusion along a line}
\author{Henri {\sc Berestycki}$^{\hbox{\small a,b,c}}$,
Jean-Michel {\sc Roquejoffre}$^{\hbox{\small d}}$, Luca {\sc Rossi}$^{\hbox{\small e,a}}$\\
\footnotesize{$^{\hbox{a }}$ Centre d'Analyse et Math\'ematiques Sociales,}
\footnotesize{EHESS - CNRS, Paris, France}\\
\footnotesize{$^{\hbox{b }}$ Department of Mathematics, University of Maryland, College Park, MD 20742, USA}\\
\footnotesize{$^{\hbox{c }}$ Senior Visiting Fellow, Inst. Advanced Studies,}
\footnotesize{HKUST, Hong Kong}\\
\footnotesize{$^{\hbox{d }}$ Institut de Math\'ematiques de Toulouse,}
\footnotesize{
Universit\'e Paul Sabatier, Toulouse, France}\\
\footnotesize{$^{\hbox{e }}$ Istituto ``G.~Castelnuovo'',}
\footnotesize{Sapienza Università di Roma, Rome, Italy}}
\date{}
\maketitle

\vspace{.2cm}

\begin{abstract}
The goal of this work is to understand and quantify how a line with nonlocal  diffusion given by an integral enhances a reaction-diffusion process occurring in the surrounding plane. This is part of a long term programme where we aim at modelling, in a mathematically rigorous way,  the effect  of transportation networks on the speed of biological invasions or propagation of epidemics. 

We prove the existence of a global propagation speed and 
characterise in terms of the parameters of the system the situations where such a speed is boosted by the presence of the line.
 In the course of the study we also uncover unexpected regularity properties of the model. On the quantitative side, the two main parameters are the intensity of the diffusion kernel and the characteristic size of its support. One outcome of this work is that the propagation speed will significantly be enhanced even if only one of the two is large, thus broadening the picture that we have already drawn in our previous works on the subject, with local diffusion modelled by a standard Laplacian. 

We further investigate the role of the other parameters, enlightening some 
subtle effects due to the interplay between the diffusion in the half plane and that on the line. 
Lastly, in the context of propagation of epidemics, we also discuss the model where, instead of a diffusion, displacement on the line comes from a pure transport term.
 \end{abstract}


\bigskip




\hypersetup{linkcolor=black}
\tableofcontents
\hypersetup{linkcolor=red}


\section{Biological invasions and epidemics in the presence of a line}\label{s2}

This work is part of a programme aimed at understanding the effect of a line, or a network of lines, on  propagation properties in the context of reaction-diffusion. The underlying motivation is to model the effect of a line of transportation, such as a road, a railway, or a waterway, 
on the spreading of a biological invasion, or the dissemination of epidemics. In the present paper, we examine the case of a nonlocal diffusion on the line.

In  \cite{BRR2}, or \cite{SIRT1} we considered the
 framework of local diffusion on the line. There, we analysed the case 
of a line having a diffusion of its own, given by a multiple of the Laplacian (thus associated with random brownian motion of individuals), and coupled with a Fisher-KPP, or diffusive SIR  process 
with   diffusion in the adjacent plane.
We computed the propagation velocity, and an important outcome was that the overall propagation was increased as soon as the diffusion on the line exceeded a certain explicit threshold. Our results here recover and generalise those of the aforementioned papers.

An even more drastic effect was observed by A.-C. Coulon and the authors of the present paper in \cite{BCRR}, where the diffusion was given by a fractional Laplacian. In that case, the fronts propagate exponentially fast in time, just as   Fisher-KPP fronts with such a kind of diffusion (see Cabré-Roquejoffre~\cite{CR}).  Thus, the present work is a  further evidence that the line  communicates the characteristics of its own diffusion to the whole process, regardless of the type of diffusion in the rest of the plane. 

%

Our analysis is carried out  for two distinct models: biological invasions in the context of
population dynamics, and the spreading of epidemics. We describe these two frameworks in the following two subsections. They are closely related through a well known transformation that we will recall in Section~\ref{s2bis}.

\subsection{Biological invasions}\label{sec:Bio}
In the spirit of the system we introduced in \cite{BRR2},
we describe the invasion of a species whose dissemination is enhanced by a \blue{straight}
line of transportation by considering the upper half plane $\RR^2_+=\RR\times(0,+\infty)$, that we call --~in a stylised way~-- ``the field'', while its boundary $\RR\times\{0\}$
is referred to as ``the road''.
We restrict ourselves to the upper half-plane, rather than considering the whole plane crossed by a line, because
the results in the two cases can easily be deduced from one another.
We then consider two distinct density functions describing the same species:
$u(t,x)$, $t>0$, $x\in\RR$, is the density on the road, 
$v(t,x,y)$, $t>0$, $(x,y)\in\RR^2_+$, is the one in the field. 
The individuals in the field are assumed to diffuse, according to the Laplace operator
with a diffusion coefficient  $d>0$, 
and to proliferate with rate $f(v)$, that we take to be smooth
 and of the Fisher-KPP type, i.e.~satisfying
\begin{equation}\label{KPP}
	f(0)=f(1)=0,\qquad 0<f(s)\leq f'(0)s \ \text{ for all }\,s\in(0,1).
\end{equation}
For the sake of definiteness, we extend $f(s)$ by a negative function for values of  $s>1$.

In contradistinction to what happens in the field, we assume that the diffusion on the road is {\em nonlocal}, reflecting 
displacements of larger amplitude (see Turchin \cite{T} for a biological discussion).   For this purpose, we let $K:\RR\to\RR$ be an even, 
smooth, nonnegative  function with unit mass, supported in $[-1,1]$. For $L>0$, we set
\begin{equation}
\label{e0.1}
K_{L}(x):=\frac{1}LK\Big(\frac{x}L\Big).
\end{equation}
Thus $K_{L}$ is supported in $[-L,L]$ and the mass condition is preserved:
$$\int_\RR K_L(x)dx=1.$$
We define the nonlocal diffusion operator, depending on the parameter $L>0$, by
$$
\J u(x): =\int_\RR K_L(x-x')\big(u(x')-u(x)\big)dx'.
$$
The parameter $L$ then represents the order of the distance that individuals travel to owing to this 
\blue{nonlocal} dispersal.

The model writes
\begin{equation}
\label{e1.2}
\left\{
\begin{array}{rll}
\partial_tv-d\Delta v\,=&f(v) &\quad(t>0,\ (x,y)\in\RR^2_+)\\
-d\partial_yv\,=&\mu u-\nu v &\quad(t>0,\ x\in\RR,\ y=0)\\
\partial_t u-D\J u\,=&\nu v-\mu u &\quad(t>0,\ x\in\RR,\ y=0),
\end{array}
\right.
\end{equation}
where $d,D$ are positive parameters (recall that $u$ is independent of $y$).

In~\cite{BRR2} we considered the case where the diffusion on the road was given  by $D\partial_{xx}$
instead of~$D\J$. The main result 
we derived is
the existence of a spreading velocity in the horizontal direction, and the comparison with the classical
Fisher-KPP velocity.
We will review these results in Section~\ref{sec:bio}, and we will relate them with the results 
we get here on the new Model~\eqref{e1.2}.
 
\subsection{Propagation of epidemics}
The basic system of epidemiology describes bulk quantities, with total populations variables only depending on time,  that is, with no spatial dependence.         
Proposed first by Kermack-McKendick \cite{KMK} as
an elaborate set of integral equations, this model reduces, when some parameters are assumed to be constant, to the classical $SIR$ system. 
We owe it to Kendall to have first envisioned
that an epidemic could propagate as a front in space, with a definite speed.
 Kendall developed this idea in an answer to a study of Bartlett \cite{Bart}, the underlying mechanism being the nonlocal contamination rate (see also \cite{Kend}). We refer the reader to Ruan \cite{Ruan} for
 more information on the modelling issues.

%
In  \cite{SIRT1} we proposed a new stylized model,
built on the ideas of our preceding road-field model~\cite{BRR2}, 
to couple a classical $SIR$\,-type model with spatial diffusion in the plane 
with a diffusion-exchange equation on the  $x$-axis. The latter models a road on which infected individuals can travel, 
the diffusion being local and precisely given by~$-D\partial_{xx}$. 
The contamination process takes place outside the road, 
where the diffusion process of the infectious is modelled by a Laplacian $-d\Delta$.  
In addition to the classical compartments $S, I$ and $R$, this model therefore involves a fourth compartment, $T$, of traveling infectious population. Hence our reference to this model as the  ``$SIRT$--model''.

 The main conclusion involved 
 a parameter $R_0$, known as the ``Pandemic Threshold'' in such models. We showed that it  acts as a basic reproduction number: when it is larger than~$1$, there is front propagation. 
Moreover, the value of the ratio $D/d$ determines whether propagation occurs 
at the usual $SIR$ speed (as determined, for instance, by Aronson \cite{A}, or Diekmann \cite{Diek}), 
or rather at a larger speed. We also pointed out situations where the propagation velocity can be quite large, even though the epidemics looks fairly mild when $R_0$ very close to 1.

In the present paper we examine what happens when the diffusion on the road is governed by a
 nonlocal operator $\mc{J}$ as in the previous subsection. 
This translates the fact that individuals have the potential to instantly make large --~yet finite~-- jumps, 
and will result in a richer set of parameters to analyse. 
Specifically, we let $S(t,x,y)$ denote the fraction of susceptible individuals at time $t\geq0$ and position $(x,y)$
of the domain $\RR^2_+$
(as before we restrict to the upper half-plane in view of symmetry reasons).
We assume that susceptibles do not  diffuse, and let~$I(t,x,y)$ be the fraction of infected individuals in the domain,
and $T(t,x)$ (standing for ``travelling infected'') be the fraction of infected individuals on the $x$-axis. 
The former ones are assumed to diffuse according to standard diffusion with amplitude  $d>0$,
whereas the latter ones diffuse according to the nonlocal operator~$\mc{J}$ with coefficient~$D>0$. 

 Thus, the model 
writes
\begin{equation}
\label{e1.1}
\left\{
\begin{array}{rll}
\partial_tI-d\Delta I+\alpha I\,=&\beta SI &\quad(t>0,\ (x,y)\in\RR^2_+)\\
\partial_tS\,=&-\beta SI &\quad(t>0,\ (x,y)\in\RR^2_+)\\
-d\partial_yI\,=&\mu T-\nu I &\quad(t>0,\ x\in\RR,\ y=0)\\
\partial_t T-D\J T\,=&\nu I-\mu T &\quad(t>0,\ x\in\RR,\ y=0).
\end{array}
\right.
\end{equation}
 Of specific  interest to us here will be to determine the spreading velocity for this new model, and to compare it with the speed without the road. We will then study its dependence on the parameters
 of the system, in particular the ones involved in the nonlocal diffusion:~$D$,~$L$.
 
\blue{Note that, as in our earlier work \cite{SIRT1}, this system does not involve a spatial diffusion mechanism on the population of susceptible. Taking into account such a diffusion would lead one to replace the second equation in System~\eqref{e1.1} by 
 \begin{equation}\label{diffS}
\partial_tS - d_S \Delta S\,=-\beta SI \quad\text{in}\quad t>0,\ (x,y)\in\RR^2_+. 
 \end{equation}
  In the system we consider here, we make the simplification $d_S=0$ that allows us to conduct a mathematical analysis of this system. 
 As a matter of fact, even in the classical case of the $SIR$ setting with spatial diffusion on all compartments, 
 determining the asymptotic speed of diffusion is an open problem in general. For partial or related results on this question we refer the reader to the works of Ducrot \cite{DucrotSIR}, 
 Ducrot, Giletti and Matano \cite{DGM-sys}. 
 Therefore, it is customary to assume, even in this classical framework, that $d_S=0$ (compare for instance the book of 
 Murray \cite[Section~13.4]{Murray}).
From a modelling point of view, such a simplification is motivated by the consideration that the changes in the susceptible population resulting from spatial diffusion is of a lower order magnitude with respect to the size of population (in contradistinction with the diffusion of infectious). 
\newline\indent
 Nonetheless, we would like to stress that the study of System~\eqref{e1.1} with the second equation replaced by \eqref{diffS} (as well as the classical $SIR$ model) with diffusion $d_S>0$ is an open problem.}
 
 We also apply our methods to the analysis of a somewhat different, yet related, framework.  Namely, in the context of epidemic propagation we consider the influence of a line with a {\em transport} mechanism in one direction. It is of interest in situations where individuals travel away from main towns for instance as they try to move away from contamination. This effect was largely reported in the recent COVID-19 pandemic in several countries. We will see that the transport also enhances the global propagation, but that the value of the spreading speed is, unexpectedly, strictly less than that of the transport.
 
 \bigskip
  The paper is organised as follows. 
 Section~\ref{sec:results} contains the statements of the main results. 
 In Section~\ref{sec:IVP}, we perform a preliminary classical transformation that allows one 
 to reduce the $SIRT$ model~\eqref{e1.1} to a slight perturbation of~\eqref{e1.2}. Next
 we state a result about the well-posed character of the two systems.
 Section~\ref{sec:bio} is devoted to the model~\eqref{e1.2} for biological invasions, for which we present 
 firstly the Liouville-type result as well as the local-in-space
 convergence, and next the result about the propagation.
 Section~\ref{sec:speedSIRT} focuses on the specific features of the $SIRT$ model:
 the description of the steady state and its behaviour at infinity, and the result on
 the spreading of the epidemic wave.
 The results presented in those subsections are proved in Sections \ref{s2bis}, \ref{s2ter} 
 and \ref{s6} respectively. 
 In Section \ref{s7} we discuss our results and emphasise analogies and novelties
 with respect to previous models.

 \section*{\blue{Acknowledgments}}

 \blue{ 
 The authors are thankful to the referees for their helpful comments.}

\medskip

\blue{
 The research leading to these  results has received funding from the 
 ANR project ``ReaCh'' (ANR-23-CE40-0023-01),  
 from the PRIN project ``PDEs and optimal control methods in mean field games, population dynamics and multi-agent models''
 and from GNAMPA-INdAM. Part of this work was carried
 out during visits by H. Berestycki and J.-M. Roquejoffre
 to the Department of Mathematics of the Università Roma ``Sapienza'', whose hospitality is thankfully acknowledged.
}

\medskip
 


\section{Main results}\label{sec:results}

\subsection{Initial value problems}\label{sec:IVP}
 
The preliminary question is whether the Cauchy Problem for models \eqref{e1.2} and \eqref{e1.1} is well-posed.  
So, we supplement  \eqref{e1.2} with the initial datum
\begin{equation}
\label{e2.100}
(u(0,x),v(0,x,y))=(0,v_0(x,y))
\end{equation}
with $v_0\geq0$ smooth and compactly supported, 
and \eqref{e1.1} with the initial datum
\begin{equation}
\label{e2.101}
(S(0,x,y), I(0,x,y), T(0,x))=(S_0,I_0(x,y),0),
\end{equation}
with $S_0>0$ constant and $I_0\geq0$ smooth and compactly supported.
 It was  noticed in~\cite{SIRT1} that 
 \blue{the system \eqref{e1.1} can be reduced to a
 ``standard'' Fisher-KPP type model with fast diffusion on a line, 
 by considering the cumulative densities of $I$ and $T$,} thus allowing for the
treatment proposed in \cite{BRR2}. 
Namely, calling
$$
u(t,x):=\int_0^t T(s,x)ds,\qquad v(t,x,y):=\int_0^tI(s,x,y)ds,
$$
the system rewrites as
\begin{equation}
\label{e2.4}
\left\{
\begin{array}{rll}
\partial_tv-d\Delta v\,=&f(v)+I_0(x,y) &\quad(t>0,\ (x,y)\in\RR^2_+)\\
-d\partial_yv\,=&\mu u-\nu v &\quad(t>0,\ x\in\RR,\ y=0)\\
\partial_t u-D\J u\,=&\nu v-\mu u &\quad(t>0,\ x\in\RR,\ y=0),
\end{array}
\right.
\end{equation}
with
\begin{equation}\label{f(v)}
	f(v):=S_0(1-e^{-\beta v})-\alpha v,
\end{equation}
\begin{equation}\label{I0}
	S_0>0\text{ constant},\qquad I_0\not\equiv0\text{ non-negative, smooth and compactly supported},
\end{equation}
and initial datum \phantom{\eqref{f(v)}}
\begin{equation}\label{IDuv}
	(u(0,x),v(0,x,y))\equiv(0,0).
\end{equation} 
{This type of reduction is classical in the framework of the $SIR$ model in the whole space. It is the crucial observation that allowed
  Aronson \cite{A} to derive the asymptotic spreading speed.} 
\begin{Th}\label{cauchy}
The initial value problems \eqref{e1.2},\eqref{e2.100} with $v_0\geq0$ bounded and smooth, 
and \eqref{e1.1},\eqref{e2.101} with $S_0>0$ constant and $I_0\geq0$ bounded and smooth,
both have a unique
 classical, bounded solution. Moreover, first-order-in-time and second-order-in-space derivatives 
 of the solution are globally bounded and H\"older continuous.
\end{Th}

 The existence proof proceeds from fairly usual arguments. 
Less standard is the uniform bound of the derivatives for system \eqref{e1.2} and its 
compact perturbation~\eqref{e2.4},
as the line has no particular smoothing effect.
It is obtained as a consequence of the maximum principle in narrow domains applied to 
the equations satisfied by the derivatives.
We point out that the the first-order estimates for~\eqref{e2.4} are essential for us because
the functions $T,I$ of the original model~\eqref{e1.1} 
correspond~to the time-derivatives of $u,v$.
Theorem \ref{cauchy} is proved in Section \ref{s2bis}.


\subsection{Biological invasions: steady states and propagation}\label{sec:bio}
 
Let us focus on the model for biological invasions introduced in Section~\ref{sec:Bio}, under the standing assumptions
made there.
 The first issue to understand is the classification of the~steady states of the systems and their attractiveness. 
There is an obvious 
positive solution, and the game is to show that it is globally~attractive.
\begin{Prop}
\label{t2.2}
The unique non-negative, bounded steady solutions for \eqref{e1.2} are the constant ones
$(u\equiv0,v\equiv0)$ and $(u\equiv\frac\nu\mu,v\equiv1)$. 

Moreover, any solution  $(u(t,x),v(t,x,y))$ to~\eqref{e1.2},\eqref{e2.100} 
with  $v_0\geq0,\not\equiv0$ smooth and compactly supported, converges to $(\frac\nu\mu,1)$ as $t\to+\infty$, 
locally uniformly in $x\in\RR$,~$y\geq0$.
\end{Prop}

Proposition \ref{t2.2} is proved in Section \ref{Liouville} following the same scheme as for the 
model with local diffusion on the road considered in~\cite{BRR2}. The arguments require uniform regularity
of the solutions, which is guaranteed here by Theorem~\ref{cauchy}.
In the realm of biological invasions,
the statement about the long-time behaviour 
of the solution is known as the ``hair trigger effect'' 
(the terminology dates back to Aronson-Weinberger~\cite{AW}).

Once the local behaviour of the solution is established, one naturally turns to the study of the propagation.
For the classical Fisher-KPP equation 
\begin{equation}
\label{e2.10}
v_t-d\Delta v=f(v),\ \ \ t>0,\ X\in\RR^N,
\end{equation}
propagation occurs with an {\em asymptotic spreading speed} $\cK$, which is explicit: 
$\cK=2\sqrt{df'(0)}$, see Aronson-Weinberger \cite{AW}. 
The next result identifies an asymptotic spreading speed for~\eqref{e1.2}. 
\begin{Th}\label{t2.1}
 Let $(u(t,x),v(t,x,y))$ be the solution to~\eqref{e1.2},\eqref{e2.100} 
 with  $v_0\geq0,\not\equiv0$ smooth and compactly supported. 
	Then, there exists $c_*>0$ such that, for all $\e>0$, it holds
\begin{equation}\label{spreading}
\sup_{\vert x\vert\leq(c_*-\e)t}\big|(u(t,x),v(t,x,y))-
		\Big(\frac\nu\mu,1\Big)\big|\to0,\ \ 
		\sup_{\vert x\vert\geq(c_*+\e)t}\big|(u(t,x),v(t,x,y))\big|\to0,
\end{equation}
		as ${t\to+\infty}$, locally uniformly with respect to $y\geq0$.
		
		In addition, there is a quantity $D_*>0$ such that the spreading speed $c_*$ satisfies
		$$c_*\begin{cases} = \cK & \text{if }D\leq D_*\\
							> \cK & \text{if }D>D_*
				\end{cases}
				\qquad\text{with }\;\cK:=2\sqrt{df'(0)}.$$	

Finally, $c_*/\sqrt{DL^2}$ converges to a positive constant as $DL^2\to+\infty$.
\end{Th}

\blue{The spreading speed $c_*$ is characterised by an implicit formula involving eigenvalues
of the operator $\J$ associated with exponential functions (cf. \eqref{e3.2} and \eqref{e4.3} below)}. 
 When the diffusion on the line is $D\partial_{xx}$, we derived
 the same type of threshold result in~\cite{BRR2}, with $D_*=2d$.
 \blue{Let us mention that this result holds true
 	in the case of a curved road which is asymptotically straight in the two directions,
 	see~\cite{curved_road}.}
 In the present case, the threshold $D_*$ is given by the non-algebraic formula~\eqref{D*} below. 
 Let us emphasise that, in contrast with the local case, it depends not
 only on $d$, but also on $f'(0)$ as well as on the 
 nonlocal kernel~$\mc{J}$ and in particular on its range $L$.
 
 \blue{
 We expect that the same threshold should hold for the $SIRT$ system in the case of a curved road which is asymptotically straight in the two directions.
 Likewise, we believe that one should be able to establish an analogous threshold for a periodic curved road. In this direction, Giletti, Monsaingeon and Zhou \cite{GMZ} had considered the case of periodic exchange coefficients (on a straight line) in the framework of the Fisher-KPP equation.
 The understanding of the influence of roads of more general shape is fully open.
 }
 
 The proof of Theorem~\ref{t2.1} uses the fact that~$\mc{J}$ has a compactly supported kernel.
 We expect the same type of result to hold for kernels decaying sufficiently fast at infinity, but 
 to prove this one should dwell deeper into the arguments developed in the present paper.
The last statement  of the theorem says that, for large~$D$ and $L$,
 the propagation in the horizontal direction really occurs as if
the diffusion in the field were also given by the nonlocal operator~$\mathcal{J}$, 
the speed in that case being given by the formula~\eqref{cL} below.
Theorem~\ref{t2.1} is proved in Section~\ref{spreadingspeed}.


\subsection{Further properties of the $SIRT$ model with nonlocal diffusion}\label{sec:speedSIRT}

In the case of the Model~\eqref{e1.1} for propagation of epidemics, 
the existence of steady solutions has to be examined at the level of cumulative densities,
which satisfy system~\eqref{e2.4} \blue{with the nonlinearity $f$ defined by~\eqref{f(v)}.
Such a nonlinearity vanishes at $0$ and
it is concave.
Let us set}
\begin{equation}\label{R0}
	f'(0)=\alpha(R_0-1),\quad\text{where }\;R_0:=\frac{S_0\beta}\alpha.
\end{equation} 
If $R_0>1$, it holds that $f'(0)>0$. In particular, when $R_0>1$, 
 the function $f$ has a unique positive zero, that we call $v_*$, and thus
it fulfils both conditions in~\eqref{KPP} with ``$1$'' replaced by $v_*>0$. 
The quantity $R_0$ can be viewed as the classical {\em basic reproduction number},
see for instance~\cite{Ka}. 

System \eqref{e2.4} is nothing else than~\eqref{e1.2} with the additional source term~$I_0$ in the first equation.
The Liouville-type
result and the stability of the unique positive steady solution hold true for this new system.
However, since $I_0$ is non-constant, 
the positive steady solution is in this case nontrivial.
It tends to the constant steady solution to~\eqref{e1.2} at infinity, with a given 
exponential decay, as stated in the following theorem.

\begin{Th}\label{thm:ltbSIRT}
	%
%
The problem~\eqref{e2.4}-\eqref{I0} has a unique non-negative, bounded steady solution $(u_\infty^r(x),v_\infty^r(x,y))$. 
	Such a solution satisfies
	$$u_\infty^r(x)=
	\begin{cases}
	0 & \hspace{-7pt}\text{if }R_0<1\\
	\frac\nu\mu\, v_* & \hspace{-7pt}\text{if }R_0>1
	\end{cases}
	\;+\,e^{-\kappa(x)|x|},\qquad
	v_\infty^r(x,y)=
	\begin{cases}
	0 & \text{if }R_0<1\\
	v_* & \text{if }R_0>1
	\end{cases}
	\;+\,e^{-\lambda(x,y)|x|},$$
	where, in the case $R_0>1$, 
		$v_*$ is the unique positive zero of $f=0$, and 
		$\kappa,\lambda$ fulfil  
	$$\lim_{|x|\to\infty}\kappa(x)=\lim_{|x|\to\infty}\lambda(x,y)=
	a_*,\qquad
	\text{locally uniformly in $y\geq0$},
	$$
	with
	$$0<a_*<	
	\begin{cases}
	\sqrt{\frac{\alpha}d(1-R_0)} & \text{if }R_0<1\\
	\sqrt{\frac{-f'(v_*)}d} & \text{if }R_0>1.
	\end{cases}
	$$
	
		Moreover, the solution $(u,v)$ to~\eqref{e2.4}-\eqref{IDuv}
			converges to $(u_\infty^r,v_\infty^r)$
			as $t\to+\infty$, locally uniformly in $x\in\RR$, $y\geq0$.
\end{Th}

The proof is presented in Section~\ref{sec:steady}.
As far as the Liouville-type result is concerned, the nonlocal operator $\mc{J}$
does not introduce any substantial difference in the proof,
compared with the $SIRT$ model with local diffusion on the road treated in~\cite{SIRT1}. 
The study of the decay, instead, requires one to understand the structure of 
exponential eigenfunctions for the nonlocal operator.

%
%

{
Theorem~\ref{thm:ltbSIRT} is a ``pandemic threshold theorem'' 
(as introduced by Kendall in the appendix to Bartlett's paper~\cite{Bart} p. 67)}
%
which exhibits two opposite scenarios according to whether $R_0$ is below or above the value $1$.
Indeed, since the loss of susceptible individuals at a given location $(x,y)$ throughout the whole epidemic course
 is $I_{tot}(x,y):=S_0-S(+\infty,x,y)$,
and one has that $S=S_0 e^{-\beta v}$ by
the second equation in~\eqref{e1.1}, Theorem~\ref{thm:ltbSIRT} yields 
$$I_{tot}(x,y)=\big(1-e^{-\beta v_\infty^r(x,y)}\big),$$ 
whence in particular
$$
\lim_{|(x,y)|\to\infty}I_{tot}(x,y)=\begin{cases} 
0 & \text{if }R_0\leq1\\
S_0\big(1-e^{-\beta v_*}\big) & \text{if }R_0>1.
\end{cases}
$$
This means that the epidemic wave spreads throughout the territory if and only if ${R_0>1}$.
Therefore, Model~\eqref{e1.1} displays the same well-known dichotomy as the classical epidemic models.


The next question is then to determine the speed of propagation of the epidemic in the case $R_0>1$. 
Similarly to what happens when the diffusion of the epidemic on the line is given by the Laplacian,
considered in~\cite{SIRT1},
the speed of propagation on the half-plane can be enhanced by the faster diffusion on the line.
\begin{Th}\label{thm:cSIRT}
	Assume that $R_0>1$. Let $(u,v)$ be the solution to~\eqref{e2.4}-\eqref{IDuv}. 
	Then, there exists $\cSIRT>0$ such that, for all $\e>0$, it holds
	$$\sup_{\vert x\vert\leq(\cSIRT-\e)t}\big|(u(t,x),v(t,x,y))-
		(u_\infty^r(x),v_\infty^r(x,y))\big|\to0,
		$$
		$$\sup_{\vert x\vert\geq(\cSIRT+\e)t}\big|(u(t,x),v(t,x,y))\big|\to0,
		$$
		as $t\to+\infty$, locally uniformly with respect to $y\geq0$.
		
		In addition, there is a quantity $D_*>0$ such that
		the spreading speed $\cSIRT$ satisfies
		$$\cSIRT\begin{cases} = \cSIR & \text{if }D\leq D_*\\
							> \cSIR & \text{if }D>D_*
				\end{cases}
				\qquad\text{with }\;\cSIR:=2\sqrt{d\alpha(R_0-1)}.$$	
\end{Th}
 The quantity $\cSIR=2\sqrt{d\alpha(R_0-1)}$ is the speed of propagation for the classical $SIR$ model 
 with local diffusion on the infected, without the road.
 The speed $\cSIRT$ and the threshold $D_*$ in the above theorem coincide with the spreading speed $c_*$
 and the~$D_*$
 that are provided by Theorem~\ref{t2.1} in the case where~$f$ is given by~\eqref{f(v)}.
 This is not surprising, because systems~\eqref{e2.4} and~\eqref{e1.2} only differ for
 the compactly supported source term $I_0$, which, as one may expect, does not affect
 the dynamics of the solution far from the origin.
 This is rigorously proved at the end of Section~\ref{sec:steady}.

\blue{
In conclusion, recalling that the quantities $u,v$ 
are the cumulative densities of the compartments $T,I$ of infected individuals, one can retrieve from Theorem \ref{thm:cSIRT}
the long-time behaviour for Model \eqref{e1.1}.
One infers that, when $R_0>1$, the infection spreads in the form 
of an epidemic wave
propagating with the asymptotic speed $\cSIRT$. 
In precise terms, there exist a function $\tau_*\!:\!\RR\!\to\!\RR$, 
a constant $T_*\!>\!0$ 
and a  function $I_*(y)$ locally bounded from below away from $0$, such~that 
$$
T(\tau_*(x),x)\geq T_*, \ \ \  I(\tau_*(x),x,y)\geq I_*(y),
\qquad\forall (x,y)\in\RR^2_+,
$$
and the following limit hold, uniformly in
$x\in\RR$ and locally uniformly in $y\geq0$:
\[\begin{split}
\lim_{t\to+\infty}\big(T(\tau_*(x)+t,x),I(\tau_*(x)+t,x,y)\big)&=(0,0),\\
\di\lim_{\tau_*(x)\geq t\to+\infty}\big(T(\tau_*(x)-t,x),I(\tau_*(x)-t,x,y)\big)
&=(0,0).
\end{split}\]
Moreover, the function $\tau_*$ satisfies
\begin{equation}
	\label{tau}
	\lim_{\vert x\vert\to+\infty}\frac{\tau_*(x)}{\vert x\vert}=\frac1{\cSIRT}.
\end{equation}
These properties are derived from Theorem \ref{thm:cSIRT} using the regularity in time of $T,I$ (granted by
Theorem~\ref{cauchy}). As the proof does
 not require new ingredients with respect to those in  \cite[Proposition~3.7]{SIRT1}, we only state it here informally and leave the
details to the interested reader.} 

\section{Initial value problems and a-priori bounds}\label{s2bis}
 
The study of system~\eqref{e1.2} was carried out in~\cite{BRR2} in the case where
$\J$ is replaced by~$\partial_{xx}$.
 Problem \eqref{e1.2} is almost linear, the only non-linearity being the harmless 
globally Lipschitz-continuous function~$f$.
Moreover, the system displays a monotonic structure, which yields a comparison principle:
if $(u_{0}^1(x),v_{0}^1(x,y))\leq (u_{0}^2(x),v_{0}^2(x,y))$ are two ordered, initial data 
then the corresponding solutions $(u^1,v^1)$, $(u^2,v^2)$ satisfy
$$
(u^1(t,x),v^1(t,x,y))\leq(u^2(t,x),v^2(t,x,y)),\qquad\forall t>0,\ (x,y)\in\RR^2_+.
$$
The same holds true if $(u^1,v^1)$ is a sub-solution and $(u^2,v^2)$ is a super-solution, also in the generalised 
sense\,\footnote{A sub-solution (resp.~super-solution) satisfies ``$\leq$'' (resp.~``$\geq$'') 
in the three equations in~\eqref{e1.2}; a {\em generalised} sub-solution (resp.~super-solution) is the maximum (resp.~minimum)
of a finite number of sub-solutions (resp.~super-solutions). We always require (sub,super)solutions to grow at most~exponentially
in space in order to guarantee the classical maximum principle for linear parabolic~equations.}.
One can show the comparison principle following the same arguments as in the proof of \cite[Proposition~3.2]{BRR2},
which hold true when~$\partial_{xx}$ is replaced by~$\J$. Indeed, those arguments only rely on the following 
form of ellipticity condition (which is straightforward to~check): 
\begin{equation}\label{MP}
u(x)\geq u(x_0) \quad \text{for all }x\in\RR
\quad
\implies\quad \J u(x_0)\geq0,
\end{equation}
which is the key for the maximum principle to hold,
as well as on the existence of a smooth function $\chi:\RR\to[0,+\infty)$ such that
$$\|\chi'\|_\infty,\|\chi''\|_\infty<+\infty,\quad
\lim_{x\to\pm\infty}\chi(x)=+\infty,\qquad
|\J\chi|\leq 1$$
(which can easily be constructed after noticing that $|\J\chi|\leq C L\|\chi'\|_\infty$).
The monotonic structure of~\eqref{e1.2} entails that if the initial datum  $v_0$ in \eqref{e2.100} 
is smooth (we have taken $u(0,.)\equiv 0$ for convenience) the comparison principle propagates to the derivatives and
entails {\em exponential in time} a priori bounds for the successive derivatives of $u$ and $v$.
The same is true for Model~\eqref{e1.1}, as a consequence of the fact that the integrated system \eqref{e2.4}
 also enjoys exponential in time 
 a priori bounds for its solutions. Therefore, what we need to
do is to construct solutions to \eqref{e1.2},\eqref{e2.100} and \eqref{e2.4}-\eqref{IDuv}
and to derive 
{\em uniform in time} estimates for their 
successive derivatives.


\begin{proof}[Proof of Theorem \ref{cauchy}] 
Let us first concentrate on 
the model for biological invasions \eqref{e1.2},\eqref{e2.100}.
We approximate it by adding a small local diffusion $\e\partial_{xx}$, $\e>0$, on the~line:
\begin{equation}
\label{e3.1000}
\left\{
\begin{array}{rll}
\partial_tv-d\Delta v=&f(v)\quad(t>0,\ (x,y)\in\RR_+^2)\\
-d\partial_yv=&\mu u-\nu v\quad(t>0,\ x\in\RR,\ y=0)\\
\partial_t u-\e\partial_{xx}u-D\J u=&\nu v(t,x,0)-\mu u\quad(t>0,\ x\in\RR,\ y=0).
\end{array}
\right.
\end{equation}
By \cite{BRR2}, 
the Cauchy Problem for \eqref{e3.1000},\eqref{e2.100} is well-posed and enjoys the comparison principle.
Let $(u^\e(t,x),v^\e(t,x,y))$ be unique classical solution. 
Since $M(\nu/\mu,1)$ is a supersolution to~\eqref{e3.1000} for any constant $M\geq1$,
taking $M$ larger than $\max(1,\sup v_0)$ one deduces from the comparison principle that
\begin{equation}
\label{e3.1001}
0\leq u^\e(t,x)\leq M\frac \nu\mu,\quad \ 0\leq v^\e(t,x,y)\leq M,\quad
\ \forall\e>0.
\end{equation}
In order to get estimates that are uniform both in $\e$ and in $t$
%
we proceed as follows. 
Let us call $(U,V):=(\partial_xu^\e,\partial_xv^\e)$, where we have dropped the $\e$'s to to alleviate the notation.
This pair solves the linear problem
\begin{equation}
\label{e3.1002}
\left\{
\begin{array}{rll}
\partial_tV-d\Delta V=&f'(v^\e)V\quad(t>0,\ (x,y)\in\RR_+^2)\\
-d\partial_yV=&\mu  U-\nu V\quad(t>0,\ x\in\RR,\ y=0)\\
\partial_t U-\e\partial_{xx}U-D\J U=&\nu V(t,x,0)-\mu U\quad(t>0,\ x\in\RR,\ y=0).
\end{array}
\right.
\end{equation}
Pick any $\ell>0$, having in mind that we will require it to be small. 
By interior parabolic regularity, see e.g.~\cite[Theorem 9.10.1]{Lad}, 
applied to the Fisher-KPP equation in~\eqref{e3.1000} (recall that $f$ is smooth),
for given $\alpha\in(0,1)$ there is $C_\ell>0$ 
such that
$$
\vert V(t,x,y)\vert\leq C_\ell\big(\Vert v^\e\Vert_{L^\infty(\RR_+\times\RR_+^2)}
+\|v_0\|_{C^{2+\alpha}(\RR_+^2)}\big)
=: C_\ell',\qquad \hbox{$\forall t\geq0$, $x\in\RR$, $y\geq\ell$.}
$$
We have that $C_\ell,C_\ell'$ are independent of $\e$ because $v^\e\leq M$.
We introduce the pair
\begin{equation}
\label{e3.1003}
(\bar U(x),\bar V(x,y))=(1,\frac\mu\nu\cos\,\Big(\frac{\pi y}{4\ell}\Big)).
\end{equation}
By direct computation one checks that it satisfies third equations of \eqref{e3.1002} exactly, 
while it is a super-solution to the second one. For the first one we have that
\begin{equation}
\label{barV}
\begin{array}{rll}
-d\Delta \bar V-f'(v^\e)\bar V=&\biggl(\di\frac{\pi^2}{16\ell^2}-f'(v^\e)\biggl)\bar V,
\end{array}
\end{equation}
which is positive  for $y\in(0,\ell)$ as soon as 
\begin{equation}
\label{e3.1010}
\frac{\pi^2}{16\ell^2}>\max_{[0,M]}f'.
\end{equation}
Therefore, with such a choice of $\ell$, $(\bar U,\bar V)$ is a super-solution to~\eqref{e3.1002} 
and $-(\bar U,\bar V)$ is a sub-solution. Moreover
we~have
$$
\bar V(x,\ell)=\di\frac\mu{\nu\sqrt2}.
$$
Set
$$
M':=\sqrt 2\frac\nu\mu\,\max\big(\Vert \partial_x v_0\Vert_\infty,C_\ell'\big),
$$
whence 
$$\forall t\geq0,\ x\in\RR,\ y\in[0,\ell],\quad
M'\bar V(x,y)\geq M'\bar V(x,\ell)\geq \max\biggl(\Vert \partial_x v_0\Vert_\infty,|V(t,x,\ell)|\biggl).$$
Recall, on the other hand, that $u_0\equiv0\leq M'\bar U$.
We can then apply the comparison principle between $(U,V)$ and $\pm M'(\bar U,\bar V)$
in the strip $\RR\times(0,\ell)$ and derive the following bounds for $U,V$ in their domains of definition: 
$$
|U|\leq M'\bar U,\qquad |V|\leq M'\bar V.
$$
We have thereby shown that the functions $\partial_xu^\e(t,x),\partial_xv^\e(t,x,y)$
are bounded uniformly in $\e>0$, $t\geq0$, $x\in\RR$, $y\geq0$.

 A bound for $\partial_{xx}u^\e,\partial_{xx}v^\e$ is derived in a similar way. 
Indeed, the couple $(\partial_{xx}u^\e,\partial_{xx}v^\e)$ solves a
system quite similar to \eqref{e3.1002}, up to the fact that the first equation has the additional
inhomogeneous term $f''(v^\e)(\partial_xv^\e)^2$, which we know to be bounded independently of $\e$. 
As a consequence, recalling~\eqref{barV} and taking $\ell$ satisfying~\eqref{e3.1010},
one sees that, for a possibly larger $M'$ than before,
the pair $M'(\bar U,\bar V)$ is a super-solution to this new system
 for $y\in(0,\ell)$. 
 \\
 Now that we know that $\partial_{xx}u^\e$ is bounded, we can apply on one hand 
 the regularity theory for the oblique derivative problem in~\eqref{e3.1002}
  (see \cite[Theorem 5.18]{Lie}) and infer that the first-order-in-time and second-order-in-space derivatives of $v^\e$ are
 globally bounded and H\"older continuous, uniformly  in $\e>0$, $t\geq0$.
 On the other hand, we directly derive from the last equation in~\eqref{e3.1000} 
 the uniform $L^\infty$ bound for~$\partial_tu^\e$.
\\
One can bootstrap the above arguments, thanks to the smoothness of~$f$, and get uniform $L^\infty$ bounds on the $x$-derivatives of any order of $u^\e,v^\e$,
and also of $\partial_t u^\e,\partial_t v^\e$.
The last equation in~\eqref{e3.1000}  eventually yields that $\partial_t u^\e$
is uniformly H\"older continuous too.
%
%
%
%
These uniform estimates allow us to pass to the limit as $\e\to0$ (up to subsequences)
in~\eqref{e3.1000} and get a solution to~\eqref{e1.2}. The uniform-in-time bounds on the derivatives are 
inherited by such a solution.

 The argument for system \eqref{e2.4} is analogous, up to the fact that we have the additional term 
$I_0$ in the right-hand side for $v$, which does not affect the previous analysis because~$I_0$
is bounded and smooth. 
Once the bounds for all derivatives of $u$, $v$ are secured, 
the bounds for $S$, $I$, $T$ follow.
\end{proof}
\begin{Rk} The boundedness of the derivatives of $(u,v)$ may look rather simple, but it relies on two deep facts. The first one is that 
the diffusion process in the upper half plane transfers some regularity to the line through the exchange condition. The second is the maximum principle in narrow domains  
(which is equivalent to the existence of a positive strict super-solution, see \cite{BNV} for a very general study), 
which allows $L^\infty$ estimates even though the domain is unbounded in one direction.\\
Using classical regularisation techniques, one can relax the smoothness assumption 
on the initial datum $v_0$ and just require it to be continuous; in such a case the uniform bounds on the derivatives
hold starting from any given positive time $T$.
\end{Rk}
To conclude this section, let us dwell a little more on the issue of the boundedness of derivatives, 
and compare it to what happens in a classical reaction-diffusion equation with nonlocal diffusion, that is
\begin{equation}
\label{e3.1005}
u_t-D\mathcal{J}u=f(u),\ \ t>0, \ x\in\RR,
\end{equation}
where we assume, to fix ideas, that $f(0)=f(1)=0$,  
and that either~$f$ fulfils the Fisher-KPP condition in~\eqref{KPP}, or $f$ is bistable with
positive mass (i.e. $f'(0),f'(1)<0$, $f$ has one zero in $(0,1)$ and $\int_0^1 f>0$).
 The Cauchy Problem for \eqref{e3.1005} with initial data $0\leq u_0\leq1$ produces classical solutions, that are bounded but 
 have derivatives that may grow exponentially  in time, and moreover there is no regularisation mechanism as in parabolic equations
 (i.e.~when ``$\J$'' is replaced by~``$\Delta$'').  

 When $f$ is a bistable non-linearity, 
let us show that the preservation of regularity 
for~\eqref{e3.1005} occurs in case a travelling wave with positive speed, connecting 0 to 1, exists. 
Let $\phi(\xi)$ be such a wave.
It solves
\begin{equation}
\label{e3.10020}
\begin{array}{rll}
-D\mathcal{J}\phi+c\phi'=f(\phi)\quad\text{in }\;\RR\\
\phi(-\infty)=1,\quad\phi(+\infty)=0.
\end{array}
\end{equation}
See Bates, Fife, Ren, Wang \cite{BF} for sufficient existence conditions. One readily sees that the 
positivity of the speed $c$ yields the regularity of $\phi$, and it is this exact same property  that will trigger the regularity mechanism. Assume, in order to simplify the argument,
that the initial datum $u_0$ for \eqref{e3.1005} 
satisfies $u_0(-\infty)=1$, $u_0(+\infty)=0$ 
(if $u_0$ were compactly supported one would need $u_0$ to be sufficiently large on a large interval
and one would also have to deal with leftwards propagating waves). 
A word by word adaptation of the celebrated Fife-McLeod argument \cite{FML} shows the existence 
of two positive numbers $q$ and $\omega$, as well as two real numbers $\xi_1\geq\xi_2$ 
such that
\begin{equation}
\label{e3.10021}
\phi(x-ct+\xi_1)-qe^{-\omega t}\leq u(t,x)\leq \phi(x-ct+\xi_2)+qe^{-\omega t}.
\end{equation}
 The function $U(t,x):=-\partial_x u(t,x)$ solves
\begin{equation}
\label{e3.1006}
U_t-D\mathcal{J}U=f'(u)U,\ \ t>0,\ x\in\RR,
\end{equation}
and the underlying mechanism of Theorem \ref{cauchy} is not present here. 
As a matter of fact, when there is no diffusion, that is, $D=0$, $u(t,x)$ tends to a step function as $t\to+\infty$,
 while $U(t,x)$ grows unboundedly in time, as it becomes a sum of Dirac masses as $t\to+\infty$. 
 Yet, when $D>0$, one recovers the boundedness of $U(t,x)$, but what makes it work is that, for every $x$ in $\RR$, 
 the function $f'(u(t,x))$ is non-negative for a set of times that has bounded measure. Indeed, we may rewrite
\eqref{e3.1006} as
\begin{equation}
\label{e3.1007}
U_t+(D-f'(u(t,x)))U=DK_L*U(t,\cdot)=DK_L'*u(t,\cdot),
\end{equation}
with the latter term going pointwise to $0$ and $f'(u(t,x))\to f'(1)<0$ as $t\to+\infty$,
due to~\eqref{e3.10021}.
Therefore, the Gronwall Lemma gives, for $t$ larger than any given $T>0$, a bound of the form
\begin{equation}
\label{e3.10022}
\vert U(t,x)\vert\leq \sup_{t\geq T}|K_L'*u(t,\cdot)|+
C\mathrm{exp}\biggl(-tD+\int_0^tf'(u(s,x))\biggl)ds,
\end{equation}
for some large $C>0$. Given any $x\in\RR$, estimate \eqref{e3.10021} shows that the time spent by $u(s,x)$ in the zone $\{f'\geq0\}$ is bounded independently of $x$: this ensures the uniform boundedness for the exponential in \eqref{e3.10022}.
This argument, by the way, provides an alternative, and somewhat quicker, proof of a result by Chen \cite{Ch}
 asserting that $u(t,x)$ converges to a travelling wave. 

 When $f$ is of the Fisher-KPP type, there are no such bounds as \eqref{e3.10021}. The analogue of such bounds are given by (Graham \cite{Gr}), but more work is needed to achieve the regularity proof, doing it is outside our scope here. 
Let us notice that regularity results have been proved for travelling fronts of equations of the type \eqref{e3.1005} that are inhomogeneous in $t$ or $x$. Let us for instance mention Coville, Davila, Martinez \cite{CDM} or Shen, Shen \cite{ShSh} for particular cases of transition fronts for inhomogeneous Fisher-KPP non-linearities. All these results are different from Theorem \ref{cauchy} in spirit. This latter theorem presents indeed a new mechanism for the preservation of regularity.


\section{The biological invasions model}\label{s2ter}
\subsection{Steady states and invasion}\label{Liouville}

This section is devoted to the proof of Proposition \ref{t2.2} concerning the biological invasions model~\eqref{e1.2}.
It contains two separate statements. 
The first one is a Liouville-type result asserting that 
the only non-negative, nontrivial, bounded, steady solution for~\eqref{e1.2} 
is the constant pair $(u\equiv\nu/\mu,v\equiv1)$ 
(which is indeed a solution, as one can directly check).
The second one describes the long-time behaviour
for the Cauchy problem, locally in space.
%
%
The proof is a straightforward adaptation of 
the one for the local problem given in \cite{BRR3},
which is based on the comparison principle and a variant of the sliding method.
We give it here for the sake of completeness.

 \begin{proof}[Proof of Proposition~\ref{t2.2}]
 We simultaneously show  the two statements of the Proposition. 
 For $R>0$ sufficiently large, the principal eigenfunction $\phi$ of $-\Delta$ in the ball~$B_R$, 
 with Dirichlet boundary condition,
 satisfies $-\Delta\phi\leq f'(0)\phi$, hence  
%
  $-\Delta(\delta\phi)\leq f(\delta\phi)$ for $\delta>0$ smaller than some $\delta_0>0$.
  We extend $\phi$ by $0$ outside $B_R$, and we call $\tilde\phi(x,y):=\phi(x,y-R-1)$, so that its support does not intersect the
  line~$\{y=0\}$. We deduce that the pair
  $(0,\delta\tilde\phi)$ is a {\em generalised} steady sub-solution to~\eqref{e1.2}.
  On the other hand, the constant pair $M(\nu/\mu,1)$ is a super-solution to~\eqref{e1.2}
  for any $M\geq1$.
  Let $(\underline u,\underline v)$ and $(\overline u,\overline v)$ be the solutions to~\eqref{e1.2} 
  emerging  respectively from the initial data
  $(0,\delta\tilde\phi)$ and $M(\nu/\mu,1)$, with $\delta\in(0,\delta_0]$ and $M\geq1$.
  	Since system~\eqref{e1.2} fulfils the comparison principle --~c.f.~the beginning of Section~\ref{s2bis}~--
    	one deduces that these solutions are respectively
      non-decreasing and non-increasing in $t$, and satisfy~$(\underline u,\underline v)\leq(\overline u,\overline v)$
  for all $t>0$ and $(x,y)\in\RR^2_+$.  
  Moreover, using the boundedness of derivatives asserted by Theorem~\ref{cauchy},
  one infers that these solutions converge locally uniformly in space, as $t\to+\infty$, 
  to a steady solution
  $(\underline u_\infty,\underline v_\infty)$ and $(\overline u_\infty,\overline v_\infty)$ respectively,
  which satisfy 
  $$(0,\delta\tilde\phi)\leq(\underline u_\infty,\underline v_\infty)
  \leq(\overline u_\infty,\overline v_\infty)\leq M\Big(\frac\nu\mu,1\Big),
  \quad\forall (x,y)\in\RR^2_+.$$
  As a consequence, if we show that the steady states $(\underline u_\infty,\underline v_\infty)$ and
  $(\overline u_\infty,\overline v_\infty)$ coincide,
  we would have that any solution with an initial datum lying between 
  $(0,\delta\tilde\phi)$ and $M(\nu/\mu,1)$, for some $\delta,M>0$, converges  as $t\to+\infty$
  to such a steady state, locally uniformly in space.
  This would immediately yield the Liouville-type result, since any non-negative,  
     bounded, steady solution $(u,v)\not\equiv(0,0)$ to \eqref{e1.2}
     satisfies $v>0$ in $\RR^2_+$ (because otherwise $v\equiv0$ by the elliptic strong maximum principle and thus
     $u\equiv0$ by the second equation in~\eqref{e1.2}) hence
     $(0,\delta\tilde\phi)\leq (u,v)\leq M(\nu/\mu,1)$ for $\delta\ll1$, $M\gg1$. 
  Analogously, one would also derive the second statement of Proposition~\ref{t2.2} thanks to the parabolic
  strong maximum principle applied to $v$.
    
    To conclude the proof we then need to show that $(\underline u_\infty,\underline v_\infty)\equiv
    (\overline u_\infty,\overline v_\infty)$. We know that $(\overline u_\infty,\overline v_\infty)$
 is independent of $x$. We now show that the same is true for $(\underline u_\infty,\underline v_\infty)$,
 using a variant of the sliding method.
 Namely, since $\delta\tilde\phi\leq\underline v_\infty$ in the ball  $B_R(0,R+1)$,
 and the strict inequality holds on its boundary,
 the elliptic strong maximum principle implies that the inequality is strict in the interior too.
 Hence we can find $H>0$ such that $\delta\tilde\phi(x+h,y)\leq\underline v_\infty(x,y)$
 for all $h\in[-H,H]$ and $(x,y)\in\RR^2_+$. Since the solution emerging from $(0,\delta\tilde\phi(x+h,y))$
 converges to $(\underline u_\infty(x+h),\underline v_\infty(x+h,y))$ as $t\to+\infty$
 (by the horizontal invariance of the system), 
 it follows from the comparison principle that 
 $(\underline u_\infty(x+h),\underline v_\infty(x+h,y))\leq(\underline u_\infty,\underline v_\infty)$ in
 $\RR^2_+$. This being true for all $h\in[-H,H]$, we conclude that $(\underline u_\infty,\underline v_\infty)$
 is $x$-independent.
 Finally, since $(\underline u_\infty,\underline v_\infty)$ and $(\overline u_\infty,\overline v_\infty)$
 do not depend on $x$, the term $\J u$ in~\eqref{e1.2} drops and one ends up in the local case.
 Namely, one directly applies \cite[Proposition 3.1]{BRR3} (with $\rho=0$) and infers that 
 $(\underline u_\infty,\underline v_\infty)\equiv(\overline u_\infty,\overline v_\infty)$.
%
%
%
\end{proof}

\subsection{A benchmark: Fisher-KPP front propagation with nonlocal diffusion}\label{sec:Benchmark}

In order to investigate the propagation for the model \eqref{e1.2}, 
we start with considering the Fisher-KPP equation alone in the one-dimensional space, with nonlocal diffusion:
\begin{equation}
\label{e3.100}
u_t-D\mathcal{J}u=f(u)\quad(t>0,\ x\in\RR).
\end{equation}
It is well known (see, for instance, Liang-Zhao \cite{LZ}, Thieme-Zhao \cite{TZ}) that the speed of propagation 
for compactly supported initial data is inferred, in this case, from the study of the plane waves of \eqref{e3.100}, linearised around $u=0$:
\begin{equation}
\label{e3.101}
u_t-D\mathcal{J}u=f'(0)u\quad(t>0,\ x\in\RR)
\end{equation}
	Plane waves for \eqref{e3.101} are sought for in the exponential form
\[
u(t,x)=e^{-a(x-ct)},
\]
with $a,c>0$.
Direct computation shows that
\[
\begin{split}
	\J u &=\di\int_\RR K_L(x-x')\big(e^{-(a x'-ct)}-e^{-(a x-ct)}\big)dx'\\
	 & = u\int_\RR K_L(x-x')\big(e^{-a(x'-x)}-1\big)dx'\\
	& =\vp_{L}(a) u,
\end{split}
\]
where we have set 
\begin{equation}
\label{e3.2}
\begin{split}
\vp_{L}(a) :&=\int_{\RR}K_L(x)(e^{a x}-1)dx\\
&=2\int_0^{+\infty}K_L(x)(\cosh(ax)-1)dx.
\end{split}
\end{equation}
Thus $\vp_L(a)$ is an eigenvalue for $\J$. The function $a\mapsto\vp_L(a)$ is analytic, even, nonnegative and vanishes at $0$. 
It further satisfies
\begin{equation}
	\label{phi''}
\vp_{L}''(a)=2\int_0^{+\infty}K_L(x)\cosh(ax)x^2dx,
\end{equation}
hence it is strictly convex.
Moreover, for all $\delta\in(0,1)$ it holds that
\begin{equation}
	\label{e5.100}
	\delta\Big(\min_{[1-\delta,1]}K\Big) \left(\frac12e^{(1-\delta)aL}-1\right)\leq\vp_L(a)\leq e^{aL}-1,
\end{equation}
which shows that $\vp_L(a)$ grows exponentially as $a\to+\infty$.

The function~$u(t,x)=e^{-a(x-ct)}$ solves \eqref{e3.101} if and only if
\begin{equation}
\label{e3.103}
c=\frac{D\vp_L(a)+f'(0)}{a}.
\end{equation}
We call $c_L(D)$ the minimal value of $c$ in~\eqref{e3.103}
as $a$ varies on~$(0,+\infty)$. By the strict convexity of $\vp_L$,
such a minimal value is attained by a unique $a=a(D)$.
It turns out that the minimum $c_L(D)$ is the asymptotic spreading speed  for~\eqref{e3.100}, as well as
the minimal speed of travelling waves, see for instance \cite{CD}.
In order to see the important effect of the road on the overall propagation in model \eqref{e1.2}, it is useful to provide an order of magnitude of $c_L(D)$ when $D$ is large, the other parameters being fixed. 
The minimiser $a(D)$ in~\eqref{e3.103} satisfies
\begin{equation}\label{a_}
\frac{f'(0)}D
=
a(D)\varphi_L'(a(D))-\varphi_L(a(D))
=\int_0^{a(D)} x\varphi_L''(x)dx.
\end{equation}
This indicates that $a(D)\to0$ as $D\to+\infty$, and more precisely that
\begin{equation}\label{aD}
\frac{f'(0)}D
=\frac12 a^2(D)\big(\vp_L''(0)+o(1)\big)\quad\text{as }\;D\to+\infty.
\end{equation}
Recalling that $c_L$ is given by~\eqref{e3.103} with $a=a(D)$ satisfying~\eqref{a_}, one gets
$$c_L(D)=D\varphi_L'(a(D))=Da(D)\big(\varphi_L''(0)+o(1)\big)\quad\text{as }\;D\to+\infty,$$
whence, by~\eqref{aD},
$$c_L(D)=\sqrt{2Df'(0)\big(\varphi_L''(0)+o(1)\big)}\quad\text{as }\;D\to+\infty.$$
Finally, computing
$$\vp_{L}''(0)=
\int_{\RR}K_L(x)x^2dx=L^2\int_{\RR}K(x)x^2dx=:L^2\langle x^2K\rangle ,$$
we eventually find
\begin{equation}\label{cL}
c_L(D)=\sqrt{2Df'(0)\big(L^2\langle x^2K\rangle +o(1)\big)}\quad\text{as }\;D\to+\infty.
\end{equation}

%

\subsection{Spreading speed }\label{spreadingspeed}

We now turn to Theorem \ref{t2.1} which asserts the existence of an asymptotic spreading speed~$c_*$
for model \eqref{e1.2}.
This will be given by the least $c$ so that the    
linearised system around $(0,0)$ has plane wave  supersolutions (i.e.~satisfying
the inequalities ``$\geq$'' in the three equations)
moving with speed $c$ in the $x$ direction.
 The analogous property is proved in \cite{BRR2} when the diffusion on the line is $-D\partial_{xx}$. 

To start with, we linearise the system \eqref{e1.2} around~$v\equiv0$~:
\begin{equation}
\label{e4.1}
\left\{
\begin{array}{rll}
\partial_t u-D\J u\,=&\nu v-\mu u &\quad(t>0,\ x\in\RR,\ y=0)\\
\partial_tv-d\Delta v\,=&f'(0)v &\quad(t>0,\ (x,y)\in\RR^2_+)\\
-d\partial_yv\,=&\mu u-\nu v &\quad(t>0,\ x\in\RR,\ y=0).
\end{array}
\right.
\end{equation}
The novelty with respect to~\cite{BRR2} is in the nonlocal term $\J u$ instead of $\partial_{xx}u$.
However, as seen in Section~\ref{sec:Benchmark}, exponential functions 
are also eigenfunctions for such a nonlocal operator. 
This is why we look for plane wave solutions for \eqref{e4.1} as exponential functions,
exactly as in the local case:
\begin{equation}
\label{expsol}
(u(t,x),v(t,x,y))=e^{-a(x-ct)}(1,\gamma e^{-b  y})\qquad 
a,\gamma,c>0, \ b\in\RR.
\end{equation}
Starting from these pairs, we will construct suitable super and sub-solutions to \eqref{e1.2}.
\begin{Lemma}\label{lem:super}
Let $\vp_L$ be defined in~\eqref{e3.2}, let $\cK:=2\sqrt{df'(0)}$ and call
\begin{equation}
\label{D*}
D_*:=\frac{2f'(0)}{\vp_L\big(\frac\cK{2d}\big)}.
\end{equation}
Then the following occur:
\begin{enumerate} 
\item if $D\leq D_*$
then system~\eqref{e4.1} admits a supersolution in the form~\eqref{expsol} if and only if $c\geq \cK$.
\item if $D> D_*$ then there exists a quantity $c_*(D,L)>\cK$ such that
the system~\eqref{e4.1} admits a supersolution in the form~\eqref{expsol} if and only if $c\geq c_*(D,L)$. 
Moreover~$c_*(D,L)$ satisfies
\begin{equation}
\label{c*DL}
\lim_{DL^2\to+\infty}\frac{c_*(D,L)}{\sqrt{DL^2}}>0.
\end{equation}
\end{enumerate}
\end{Lemma}

\begin{proof}
The third equation
of~\eqref{e4.1} rewrites in terms of the parameters in~\eqref{expsol} as $\gamma=\mu/(\nu+db)$.
This fixes $\gamma$, and entails that necessarily $b >-\nu/{d}$
(observe that also when one deals with super-solutions, it is convenient to take $\gamma$ so that
equality holds in the third equation of~\eqref{e4.1}, because increasing $\gamma$ makes the inequality $\geq$
in the first one more stringent).
The plane wave problem for~\eqref{e4.1} then reduces to the following system in the unknowns $a$ and $b $:
\begin{equation}
\label{e4.3}
\left\{
\begin{array}{rcl}
-D\vp_L(a)+ca+\di\frac{d\mu b }{\nu+db }&=&0\\
-(a^2+b ^2)+\di\frac{ca}d&=&\di\frac{\cK^2}{4d^2}\,,
\end{array}
\right.
\end{equation}
where $\vp_L(a)$ is given by \eqref{e3.2} and $\cK:=2\sqrt{df'(0)}$.
Solutions of~\eqref{e4.3} correspond in the $(a,b)$ plane, restricted to $b >-\nu/{d}$, to the intersection
between the curve $\Gamma_1$, 
given by the first equation,
and the circle $\Gamma_2$ with centre~$(\frac{c}{2d},0)$ and 
radius~$\rho(c):=\frac{\sqrt{c^2-\cK^2}}{2d}$, which is nonempty if and only if $c\geq \cK$.

Let us examine $\Gamma_1$, for given $c>0$. Recall from Section~\ref{sec:Benchmark} that $\varphi_L$ is 
analytic, even, vanishes at $0$ and it is
uniformly strictly convex (i.e.~$\inf_{\RR}\vp_L''>0$).
We find that~$\Gamma_1$ is the graph
\begin{equation}\label{G1}
b=G_1^c(a):=\frac\nu d\left(\frac{\mu}{\mu+ca-D\vp_L(a)}-1\right),
\end{equation}
which is~defined for $a\in(a_-^\infty(c,D),a_+^\infty(c,D))$ (in order to fulfil $b>-\nu/{d}$)
where $a_-^\infty(c,D)<0<a_+^\infty(c,D)$ are the solutions to
\begin{equation}
\label{e4.6}
D\vp_{L}(a_\pm^\infty(c,D))=ca_\pm^\infty(c,D)+\mu.
\end{equation}
The function $G_1^c(a)$ is analytic, has the two vertical asymptotes $a=a_\pm^\infty(c,D)$ and the two zeroes $a=0$ and
$a=a_0(c,D)$, the latter being the unique positive solution of
\begin{equation}
\label{a*}
D\vp_{L}(a_0(c,D))=ca_0(c,D).
\end{equation}
With respect to the parameter $c$, the function $G_1^c(a)$ is smooth and strictly decreasing for $a>0$ (and $a_\pm^\infty(c,D)$ 
are increasing).


Plane wave {super-solutions} to~\eqref{e4.1} correspond to the points $(a,b)$ lying in the intersection
of the disk $\mc{E}_2$ with boundary $\Gamma_2$ and the region $\mc{E}_1$ \blue{bounded from below by}~$\Gamma_1$.
It is readily seen that $\mc{E}_2$ is continuously strictly increasing with respect to~$c$, and we have seen that
the same is true for $\mc{E}_2\cap\{a>0\}$.
For $c<\cK$ the intersection $\mc{E}_1\cap\mc{E}_2$ is empty because $\mc{E}_2$ is.  
For $c=\cK$ the disk $\mc{E}_2$ reduces to its centre $(\frac{\cK}{2d},0)$
and thus $\mc{E}_1\cap\mc{E}_2\neq\emptyset$
if and only if 
\begin{equation}
\label{c*=cK}
\frac{\cK}{2d}\leq a_0(\cK,D).
\end{equation}
Therefore, if condition~\eqref{c*=cK} holds, plane wave {supersolutions} exist if and only if $c\geq\cK$.

Suppose instead that~\eqref{c*=cK} does not hold, 
that is, the disk $\mc{E}_2$ ``appears'' when $c=\cK$ outside the set $\mc{E}_1$.
Notice that the leftmost point of $\mc{E}_2$, i.e.~$\Big(\di\frac{c-\sqrt{c-\cK^2}}{2d},0\Big)$, 
approaches the origin as $c\to+\infty$.
Therefore, by the monotonicity properties of $\mc{E}_1$ and $\mc{E}_2$,
as $c$ increases starting from $\cK$, 
there has to be a first value of $c$ at which $\mc{E}_1$ and $\mc{E}_1$ intersect, 
being tangent at some point~$(a_*,b_*)$, which corresponds to a solution of~\eqref{e4.3}.
 This first~$c$ is the sought for~$c_*(D,L)$. The dichotomy is depicted in Figure~\ref{fig:c*}.

Let us look at condition~\eqref{c*=cK} in terms of $D$. 
Recall from~\eqref{a*} that $a_0(c,D)$ is the unique positive solution of $\psi(a_0(c,D))=c/D$,
with $\psi(a):=\vp_L(a)/a$.
The strict convexity of $\vp_L$ implies that the function $\psi$ is strictly increasing for $a>0$
and satisfies $\psi(0^+)=0$ and $\psi(+\infty)=+\infty$, hence $a_0(c,D)=\psi^{-1}(c/D)$ and 
then~\eqref{c*=cK} rewrites as
$$\psi^{-1}\Big(\frac\cK{D}\Big)\geq \frac\cK{2d}.$$
Rewriting $\psi(a):=\vp_L(a)/a$ and $\cK=2\sqrt{df'(0)}$ we eventually find that~\eqref{c*=cK}
is equivalent to $D\leq D_*$ with $D_*$ given by~\eqref{D*}.

 \begin{figure}[ht]
\centering
 \subfigure[$D<D_*$]
   {\includegraphics[width=5cm]{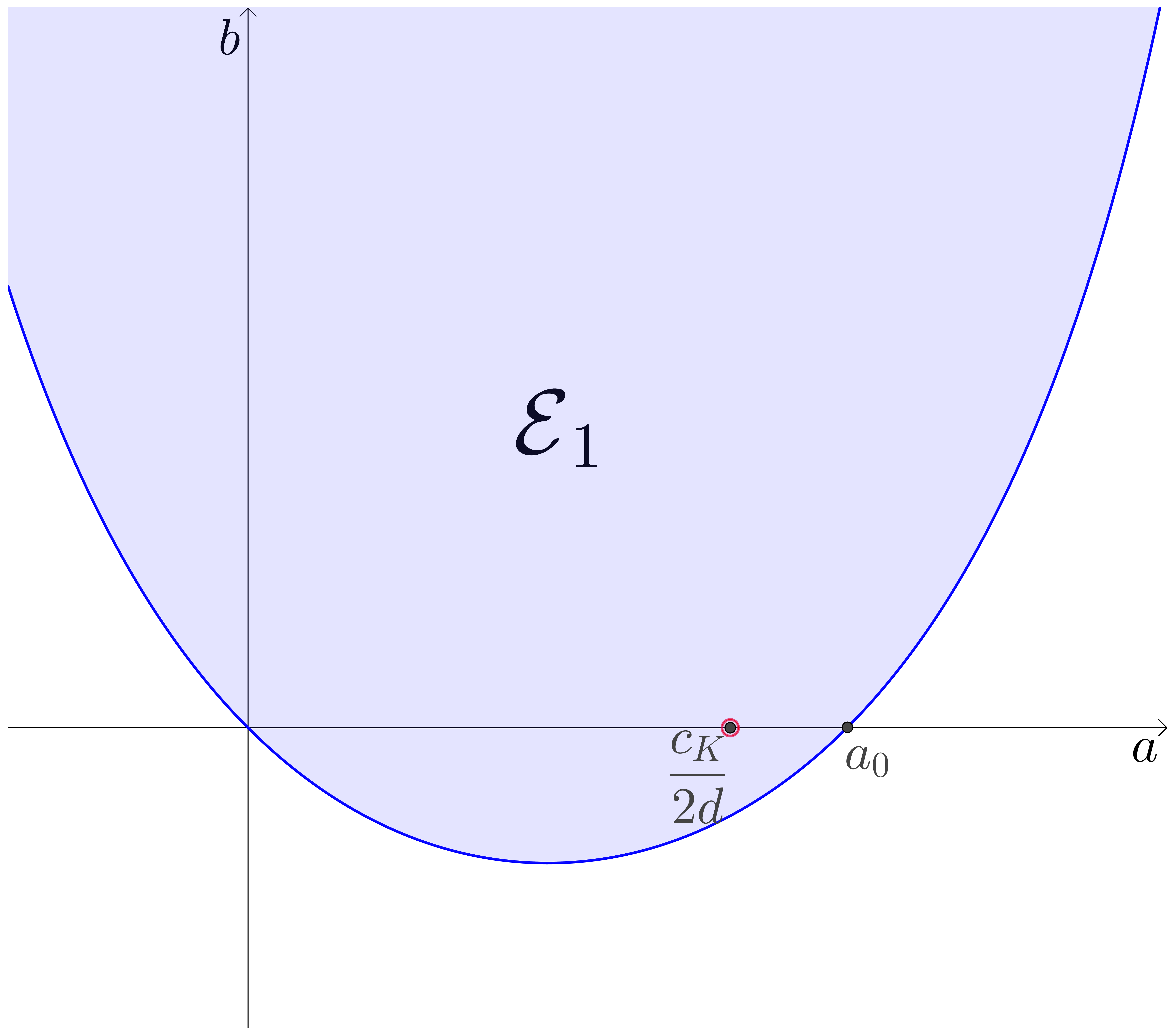}}
 \hspace{20mm}
 \subfigure[$D>D_*$]
   {\includegraphics[width=5cm]{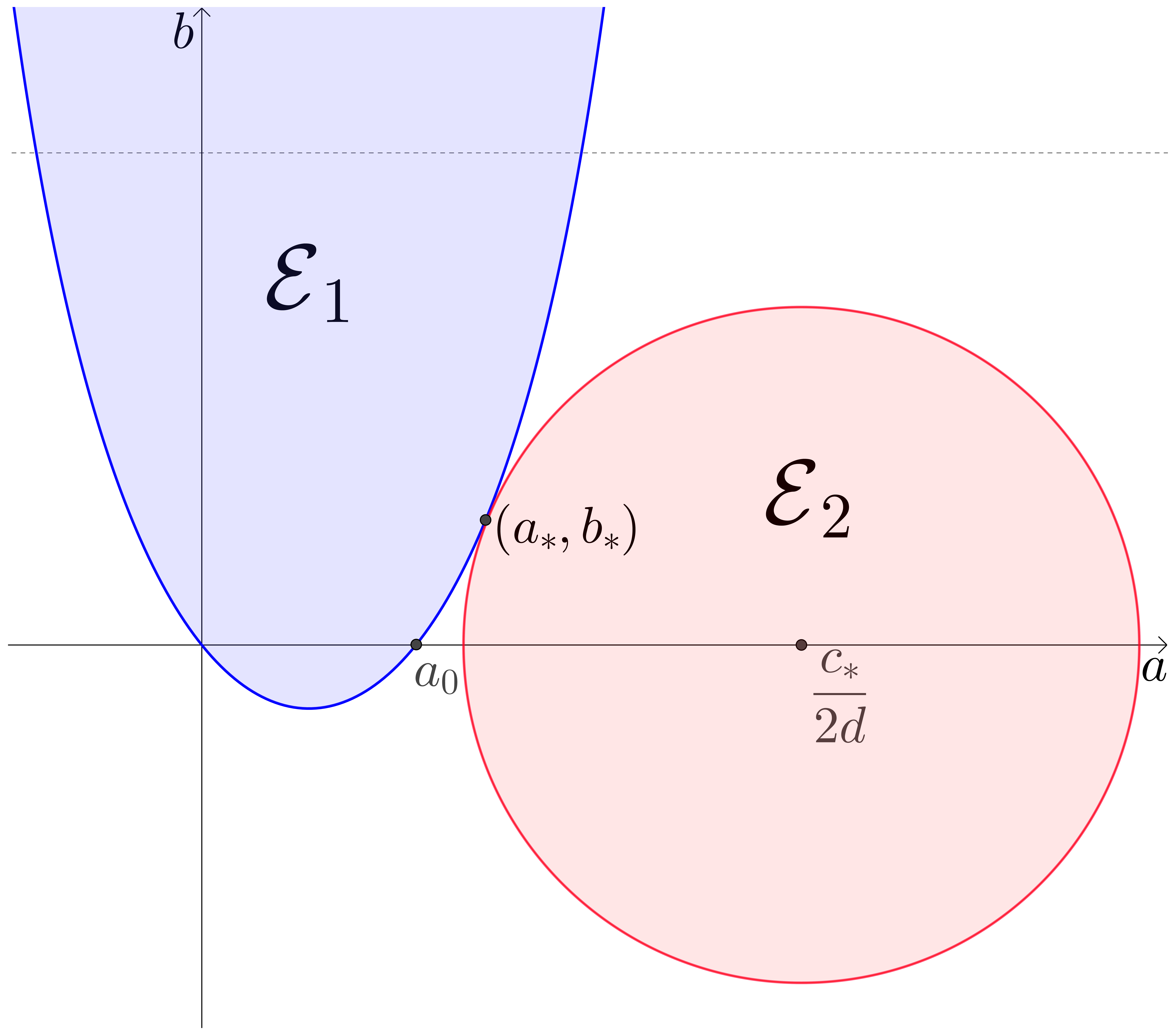}}
 \caption{The minimal $c$ for the existence of plane wave super-solutions: (a) $c=\cK$, (b) $c=c_*(D,L)$.}
 \label{fig:c*}
 \vspace{-2mm}
 \end{figure}

In order to conclude the proof of the lemma it only remains to show~\eqref{c*DL}.
Firstly, we use that $\vp_{L}(a)=\vp_1(aL)$ and that $\vp_1$ is strictly convex to derive
from~\eqref{e4.6} the existence of a positive constant $H$ such that 
$$H(a_\pm^\infty(c,D))^2 DL^2\leq ca_\pm^\infty(c,D)+\mu.$$
This yields  
\begin{equation}
\label{ainftyto0}a_+^\infty(c_*(D,L),D)\to0\quad\text{ as }\;DL^2\to+\infty.
\end{equation}
%
%
Now,  having in mind the conclusions \eqref{aD}-\eqref{cL}
of the benchmark in Section~\ref{sec:Benchmark}, we look for $a$, $b$, $c$ in \eqref{e4.3} under the~form
$$
a=\frac{\cK\alpha}{\sqrt{2dD L^2\langle x^2K\rangle }},\qquad b=\frac{\beta}{2d},
\qquad c=\cK \sqrt{\frac{D L^2\langle x^2K\rangle}{2d}}\,w.
$$
with $\alpha,\beta\in\RR$, $w>0$,
where $\langle x^2K\rangle:=\int_{\RR}K(x)x^2dx$.
 The condition $b>-\nu/d$ reads $\beta>-2\nu$.
We know from~\eqref{ainftyto0} that any solution of~\eqref{e4.3}
satisfies $a\to0$ as $DL^2\to+\infty$, hence we can write $\vp_L(a)=\frac12\vp_L''(0)a^2+o(1)$.
Recalling that $\vp_{L}''(0)=L^2\langle x^2K\rangle$,
we end up with the following reduced system,  as $DL^2\to+\infty$:
\begin{equation}
\label{e5.502}
\left\{
\begin{array}{rcl}
-f'(0)\alpha^2+o(1)+2f'(0)w\alpha+\di\frac{\mu \beta}{2\nu+\beta}&=&0\\
\di 2w\alpha-\frac{2d}{D L^2\langle x^2K\rangle }\,\alpha^2-\frac{\beta^2}{\cK^2} &=&1.
\end{array}
\right.
\end{equation}
Let us neglect for a moment the term $o(1)$ in the first line and the one in $\alpha^2$ in the second line (whose coefficient 
tends to $0$ as $DL^2\to+\infty$). We get
\begin{equation}
\label{reduced}
\left\{
\begin{array}{rcl}
\alpha^2-2 w\alpha &=& \di\frac{\mu \beta}{2\nu f'(0)+ f'(0)\beta}\\
\alpha &=& \di \frac{1+\beta^2/\cK^2}{2w}.
\end{array}
\right.
\end{equation}
The first equation describes a curve in the $(\beta,\alpha)$ plane with asymptotes
$\alpha=\blue{w}\pm\sqrt{w^2+\mu/f'(0)}$,
while the second one is a parabola. One sees that the two curves do not intersect for~$w$~small,
and do intersect for $w$ large. 
There exists then a positive minimal value of $w$, that we call $w_*(d,\mu,\nu)$, 
such that the system has a solution. From this, one eventually deduces that the minimal value of $w$
for which the complete system~\eqref{e5.502} admits solution converges to $w_*(d,\mu,\nu)$
as $DL^2\to+\infty$. Reverting to the original parameters, we have shown that
\begin{equation}
\label{w*}
\lim_{DL^2\to+\infty}\frac{c_*(D,L)}{\sqrt{DL^2}}=
\sqrt{2f'(0)\langle x^2K\rangle}\,w_*(d,\mu,\nu),
\end{equation}
that is~\eqref{c*DL}.
\end{proof}

Next, for $D>D_*$, we derive the existence of some generalised 
subsolutions, with bounded support, that move with speed slightly smaller than $c_*(D,L)$,
where $D_*$ and $c_*(D,L)$ are given in Lemma~\ref{lem:super}.

\begin{Lemma}\label{lem:sub}	
	For $D>D_*$, there is a sequence $c\nearrow c_*(D,L)$ with associated 
	pairs of continuous, nonnegative and not identically to~$0$ functions $u_c,v_c$,
	compactly supported in $\RR$ and $\RR\times[0,+\infty)$ respectively, such that
	$$k\big( u_c(x-ct), v_c(x-ct,y)\big)$$ is a generalised 
	subsolution to~\eqref{e1.2} for $k>0$ small enough. 
\end{Lemma}

\begin{proof}
The first step is to find a sequence $c\nearrow c_*(D,L)$ such that, for any of such $c$'s,
the system
\begin{equation}
	\label{penalised}
	\left\{
	\begin{array}{rll}
		\partial_t u-D\J u\,=&\nu v-\mu u &\quad(t>0,\ x\in\RR,\ y=0)\\
		\partial_tv-d\Delta v\,=&(f'(0)-\delta)v &\quad(t>0,\ x\in\RR,\ 0<y<Y)\\
		-d\partial_yv\,=&\mu u-\nu v &\quad(t>0,\ x\in\RR,\ y=0)\\
		v= &0&\quad(t>0,\ x\in\RR,\ y=Y)\\
	\end{array}
	\right.
\end{equation}
with $\delta>0$ sufficiently small and $Y>0$ sufficiently large, admits a sign-changing solution of
the form 
$$\big(\tilde u_c(x-ct),\tilde v_c(x-ct,y)\big).$$
In addition, we will have that the sets where $\tilde u_c$ and $\tilde v_c$ are positive
 have bounded connected components and satisfy
\begin{equation}
	\label{u>0}
\{x\in\RR\ :\ \tilde u_c(x)>0\}=( 2k\pi\omega_c-\pi\omega_c/2,2k\pi\omega_c+\pi\omega_c/2),\quad k\in\Z,
\end{equation}
\begin{equation}
	\label{v>0}
	\{x\in\RR\ :\ \tilde v_c(x,0)>0\}=(2k\pi\omega_c+\vartheta_c-\pi\omega_c/2,
2k\pi\omega_c+\vartheta_c+\pi\omega_c/2),\quad k\in\Z,
\end{equation}
where $|\vartheta_c|\leq \pi\omega_c$ and $\omega_c\to+\infty$ as $c\nearrow c_*(D,L)$.
We postpone this first step to the Appendix. 

Having the functions $\tilde u_c,\tilde v_c$ at hand, one needs to truncate their support.
Call $U_c:=( -\pi\omega_c/2, \pi\omega_c/2)$ and $V_c$ the connected component of $\{\tilde v_c>0\}$
such that $V_c\cap(\RR\times\{0\})=(\vartheta_c-\pi\omega_c/2,\vartheta_c+\pi\omega_c/2)$, then define
$$u_c:=\begin{cases}
\tilde u_c & \text{in }U_c\\
0 & \text{outside}
\end{cases},\qquad
v_c:=\begin{cases}
\tilde v_c & \text{in }V_c\\
0 & \text{outside}
\end{cases}.$$
Since $u_c\geq \tilde u_c$ in $V_c\cap (\RR\times\{0\})$, one has that~$v_c(x-ct,y)$ is a generalised subsolution
of the linear parabolic problem with Robin boundary condition given by the first two equations in~\eqref{penalised},
with $u=u_c$.
Instead, the first equation in~\eqref{penalised} has to be handled more carefully due to the nonlocal term,
since $\J u_c\neq \J\tilde u_c$ even in the region $U_c$ where~$u_c\equiv\tilde u_c$.
However, for $x\in U_c$, there holds that
\[\begin{split}
\J u_c(x) &=\int_{(x-L,x+L)\cap U_c} K_L(x-x')\tilde u_c(x')dx'-\tilde u_c(x)\\
&= \J \tilde u_c(x)-\int_{(x-L,x+L)\setminus U_c} K_L(x-x')\tilde u_c(x')dx',
\end{split}\]
and we have $(x-L,x+L)\subset(-\pi\omega_c/2-L,\pi\omega_c/2+L)$. 
If~$c$ is large enough so that $\pi\omega_c\geq L$,  one has that 
 the latter set is contained in $[-3\pi\omega_c/2,3\pi\omega_c/2]$, and thus, being $\tilde u_c\leq0$
in~$[-3\pi\omega_c/2,3\pi\omega_c/2]\setminus U_c$, we deduce that $\J u_c\geq\J \tilde u_c$ in~$U_c$.
We also clearly have $\J u_c\geq0$ outside $U_c$. Summing up, we have that for $c$ sufficiently large,
$(u_c(x-ct),v_c(x-ct,y))$ is a generalised subsolution to~\eqref{penalised},
hence to~\eqref{e1.2} up to multiplication by a small $k>0$.
\end{proof}

\begin{proof}[Proof of Theorem \ref{t2.1}]
Let us show that the two limits in~\eqref{spreading}
hold with $c_*:=\cK$ if $D\leq D_*$ and  $c_*:=c_*(D,L)$ if $D> D_*$,
where $D_*$ and $c_*(D,L)$ are given by Lemma~\ref{lem:super}.
By reason of symmetry in the $x$ variable, it is sufficient to derive them for $x\geq0$.
The second limit immediately follows by comparison with the plane waves provided by Lemma~\ref{lem:super}
(which are super-solutions to the nonlinear problem~\eqref{e1.2} thanks to the KPP hypothesis).
For the first one, we make use of the sub-solutions $k( u_c(x-ct), v_c(x-ct,y))$
provided by Lemma~\ref{lem:sub} in the case $D>D_*$, for $k$ sufficiently small and $c<c_*$ which can be taken 
arbitrarily close to $c_*$. In the case $D\leq D_*$, since $c_*=\cK$,
the existence of a compactly supported sub-solution $v_c$ moving with a speed $c<c_*$ 
is standard for the Fisher-KPP equation $\partial_tv-d\Delta v=f(v)$, and thus one can simply neglect 
the equations on the line and take $u_c\equiv0$.
We then decrease $k$ if need be in order to have in addition that $k( u_c(x), v_c(x,y))\leq(u(1,x),v(1,x,y))$ 
for all $(x,y)\in\RR_2^+$.
We deduce, by comparison,
$$\big(u(1+\tau,c\tau),v(1+\tau,c\tau,y)\big)\geq \big(u_c(0),v_c(0,y)\big),
\quad\forall\tau\geq0,\ y\geq0.$$
Then, applying Proposition~\ref{t2.2} to the solution with initial datum $(u_c,v_c)$ we get, always by comparison,
$$\liminf_{t\to+\infty}\Big(\inf_{\tau\geq0}\big(u(1+\tau+t,c\tau),v(1+\tau+t,c\tau,y)\big)\Big)\geq \Big(\frac\nu\mu,1\Big),
$$
locally uniformly with respect to~$y\geq0$,
from which, calling $s:=1+\tau+t$  we derive
$$\liminf_{t\to+\infty}\Big(\inf_{s\geq t+1}\big(u(s,c(s-t-1)),v(s,c(s-t-1),y)\big)\Big)\geq \Big(\frac\nu\mu,1\Big).
$$
This yields
$$\liminf_{t\to+\infty}\Big(\inf_{x\in[0,c't]}\big(u(t,x),v(t,x,y)\big)\Big)\geq \Big(\frac\nu\mu,1\Big),
$$
for any $0<c'<c$. 
We recall, on the other hand, that $\limsup_{t\to+\infty}(u,v)\leq(\nu/\mu,1)$ uniformly in space, as seen in the 
proof of Proposition~\ref{t2.2}.
The first limit in~\eqref{spreading} then follows from the fact that $c'$ and $c$ can be taken arbitrarily close to $c_*$.
\end{proof}

Let us dwell a little bit on the quantity $w_*(d,\mu,\nu)$ appearing in~\eqref{w*}. Computing the curves 
at $\beta=0$ one infers that $w_*(d,\mu,\nu)\leq1/2$.
Observe that the parameter $d$ affects $w_*(d,\mu,\nu)$ only through the term $c_K$ in the second equation, 
which modulates the opening of the parabola. One deduces that $w_*(d,\mu,\nu)$  is decreasing with respect to~$d$ 
and tends to $1/2$ as $d\to0^+$, and to the unique solution $w$ of the equation 
$w+\sqrt{w^2+\mu/f'(0)}=\frac{1}{2w}$ as $d\to+\infty$.
Notice, however, that the limit as $d\to+\infty$ has not real meaning for the model \eqref{e1.2},
because the reduction to~\eqref{reduced} is not justified since the term neglected in the second equation of~\eqref{e5.502}
becomes large as $d\to+\infty$ (and  we indeed know that $c_*(D,L)\geq\cK\to+\infty$).

%
%




\section{The $SIRT$ model for epidemics along a line with nonlocal diffusion}\label{s6}
\subsection{A benchmark: $SIR$-type model with nonlocal diffusion}
Consider a standard $SIR$ model, where the population lives on the real line $\RR$,  
with an initially homogeneous population of susceptibles, and where the infected may move according to nonlocal diffusion
and are initially confined in a bounded region. This yields the following system for the 
density of susceptibles  $S(t,x)$ and of infected~$I(t,x)$~:
\begin{equation}
\label{e4.300}
\left\{
\begin{array}{rcll}
\partial_t S& =&-\beta SI\ & (t>0,x\in\RR)\\
\partial_tI-D\mathcal{J}I &=&\beta SI-\alpha I & (t>0,x\in\RR)\\
\end{array}
\right.
\end{equation}
completed with the initial condition $(S,I)(0,x)=(S_0,I_0(x))$,
where $S_0$ is a positive constant and $I_0$ is non-negative and compactly supported. The integrated density 
$$
u(t,x):=\int_0^tI(s,x)ds
$$
solves the following nonlocal equation with non-homogeneous right $\text{hand-side}$:
\begin{equation}
\label{e4.302}
u_t-D\mathcal{J}u=f(u)+I_0(x),
\end{equation}
with $f$ given by~\eqref{f(v)}. This is the same equation as~\eqref{e3.100} with the addition of the compact perturbation~$I_0$.
If~$R_0:=S_0\beta/\alpha$ is larger than~$1$ then $f$ is of the Fisher-KPP-type.
It~turns out that $I_0$ does not affect the large time/space dynamic of the solution, and in fact the
asymptotic spreading speeds for~\eqref{e3.100} and~\eqref{e4.302} coincide. 
Namely, the spreading speed for~\eqref{e4.302}
 is the minimal $c>0$ such that the transcendental equation~\eqref{e3.103}, that now reads
$-D\vp_L(a)+ca=\alpha(R_0-1)$,
admits a solution $a>0$.

\subsection{The influence of $R_0$ and other parameters}
 
  \blue{
  Here we investigate the influence of the parameters of the system~\eqref{e1.1} on the speed~$\cSIRT$. We are particularly concerned with situations when $R_0$ is close to 1, and $D$ or $L$ are large. 
  As a preliminary step, we reduce the system to one with non-dimensional parameters.
  Namely, we introduce the non-dimensional space and time variables:
  }
\[
t=\frac\tau\alpha,\qquad (x,y)=\sqrt{\frac{d}\alpha}(\xi,\zeta),
\]
while the unknowns $T(t,x)$ and $I(t,x,y)$ are expressed as
{\[
S(t,x,y)=S_0\SST(\tau,\xi,\zeta),\qquad I(t,x,y)=S_0\II(\tau,\xi,\zeta),\qquad {T(t,x)=S_0\sqrt{\frac{d}\alpha}\TT(\tau,\xi)}.
\]
Notice indeed that $T$ is a density per unit length, while $I$ and $S$ are densities per unit surface.}
We then consider the following five non-dimensional parameters:
\[
\DD=\frac{D}{\alpha},\qquad\Lambda=L\sqrt{\frac{\alpha}d},\qquad R_0=\frac{\beta S_0}\alpha,
\qquad \bar\nu=\frac{\nu }{\sqrt{\alpha d}},\qquad \bar\mu=\frac\mu\alpha.
\]
The speed $\cSIR$ for the model without the road is expressed as
\[
\cSIR=\sqrt{d\alpha}\,\wSIR\qquad \text{with }\;\wSIR:=2\sqrt{R_0-1}
\]
(which is non-dimensional).

%

System \eqref{e1.1} is then rewritten as
\[
\left\{
\begin{array}{rll}
\partial_\tau\II -\Delta \II + \II =&R_0 \SST\II &(\tau>0,\ (\xi,\zeta)\in\RR^2_+)\\
\partial_\tau\SST=&-R_0 \SST\II &(\tau>0,\ (\xi,\zeta)\in\RR^2_+)\\
-\partial_\zeta\II =&\bar\mu \TT -\bar\nu \II &(\tau>0,\ \xi\in\RR,\ \zeta=0)\\
{\partial_\tau\TT-\DD\tilde\J\TT}=&\bar\nu \II-\bar\mu \TT
&(\tau>0,\ \xi\in\RR,\ \zeta=0),
\end{array}
\right.
\]
where the integral operator $\tilde\J$ is given by
$$
\tilde\J\TT(\xi):=\int_\RR K_\Lambda(\xi-\xi')\big(\TT(\xi')-\TT(\xi)\big)d\xi'.
$$
\blue{The initial datum is $\SST_0\equiv1$, $\II_0=I_0/S_0$, $\TT_0\equiv0$.}

\blue{
We are interested in the limit 
$R_0\to 1^+$.  For this, we adapt the computation of Section~\ref{spreadingspeed} to the $SIRT$ model. Compare in particular 
 formula~\eqref{w*}.
 For the sake of completeness and  in order to highlight the role of the new parameter $R_0$, we detail the argument.
}

The integrated quantities
$$
\UU(\tau,\xi):=\int_0^\tau\TT(\sigma,\xi)d\sigma,\qquad \VV(\tau,\xi,\zeta):=
\int_0^\tau\II(\sigma,\xi,\zeta)d\sigma
$$
will then solve
\begin{equation}
\label{e.param.5}
\left\{
\begin{array}{rll}
\partial_\tau\UU-\DD \tilde\J\UU= & \bar\nu \VV(\tau,\xi,0)-\bar\mu\UU\quad(\tau>0,\ \xi\in\RR)\\
\partial_\tau \VV-\Delta \VV=&\tilde f(\VV)+\II_0(\xi,\zeta)\quad(\tau>0,\ \xi\in\RR,\ \zeta>0)\\
-\partial_\zeta \VV(\tau,\xi,0)=&\bar\mu \UU(\tau,\xi)-  \bar\nu \VV(\tau,\xi,0)\quad(\tau>0,\ \xi\in\RR).
\end{array}
\right.
\end{equation}
The function $\tilde f$ is given by 
$\tilde f(\VV):=1-e^{-R_0\VV}-\VV$, so that, $\tilde f'(0)=R_0-1=\di\frac{\wSIR^2}4$.
The minimal reduced speed for \eqref{e.param.5}, that we call $\wSIRT$, is  the least $w$ for which the system in $a,b$
\begin{equation}
\label{e.param.7}
\left\{
\begin{array}{rll}
{-\DD \varphi_\Lambda(a)+wa+\di\frac{\bar\mu b}{\bar\nu+b}}=&0\\
-(a^2+b^2)+wa=\di\frac{\wSIR^2}4
\end{array}
\right.
\end{equation}
has solutions, \blue{where $\varphi_\Lambda(a)=\vp_{1}(a\Lambda)$ is given by \eqref{e3.2}.}

{Since $R_0$ is only slightly larger than 1, $\wSIR$ is now a small positive parameter. From the second equation of \eqref{e.param.7} we surmise  that $a$  will be much smaller than $b$. Indeed, we expect $w$ to be much larger than $\wSIR$,
while $wa$ and $b^2$ should be of the order $\wSIR^2$.  We set:}
\begin{equation}
\label{e.param.10}
a={\wSIR}\bar a,\ \ b=\wSIR\bar b,\ \ w=\wSIR\bar w,
\end{equation}
and we introduce the new parameter
\begin{equation}
\label{e.param.9}
\lambda :=\frac{\bar\mu}{\bar\nu\wSIR} \; .
\end{equation}
This leads to the final reduced system
 (a rigorous justification would be easy, {\it via} the Implicit Functions Theorem):
\begin{equation}
\label{e.param.11}
\left\{
\begin{array}{rll}
-\di\frac{\DD}{4(R_0-1)}\vp_1(2\Lambda\sqrt{R_0-1}\,\bar a)+\bar w\bar a+{\lambda\bar b}=&0\\
-\bar b^2+\bar w\bar a=\di\frac14.
\end{array}
\right.
\end{equation}
The second equation is the standard parabola $\Gamma_{2,\bar w}$
$$
\bar a=\frac1{\bar w}\biggl(\frac14+\bar b^2\biggl)=:h(\bar w,\bar b).
$$
A final observation will allow us an easy treatment of the first equation. 
We notice that the second and third terms are expected to be of finite size, so that the first one should also be of finite size.
Given that the ratio $\di\frac{\DD}{R_0-1}$ is large, the term $\vp_1(2\Lambda\sqrt{R_0-1}\,\bar a)$ should be small, 
i.e.~$\Lambda\sqrt{R_0-1}\bar a$ should be small. And so, we may approximate
$$
\frac{\DD}{4(R_0-1)}\vp_1(2\Lambda\sqrt{R_0-1}\,\bar a)\sim M_1\DD\Lambda^2\bar a^2=M_1{\frac{DL^2}d}\bar a^2,\qquad M_1=\int_0^{+\infty}\xi^2K(\xi)d\xi.
$$
Thus, with the same analysis as in Section~\ref{spreadingspeed}, we obtain the existence of a positive bounded function $\WSIRT(\lambda)$ such that 
\begin{equation}
\label{e.param.12}
\lim_{_{DL^2\to+\infty,\,R_0\to1^+}}\sqrt{\frac{d}{4M_1DL^2(R_0-1)}}\,\wSIRT=\WSIRT(\lambda),\ \ \ \lambda=\frac{\bar\mu}{\bar \nu \wSIR}.
\end{equation}

We leave the interpretation of this formula to the Discussion Section~\ref{s7}.

\subsection{Proof of the results on the $SIRT$ model}\label{sec:steady}

We start with the proofs of the Liouville-type result and of the local stability property
contained Theorem~\ref{thm:ltbSIRT}. We can follow exactly the same arguments used in the proofs of \cite[Theorem 3.4, 3.5]{SIRT1},
thanks to the fact that system~\eqref{e2.4} enjoys the comparison principle
and the uniform regularity of the solutions, due to Theorem~\ref{cauchy}. We sketch these arguments below
for completeness.

\begin{proof}[Proof of  Theorem~\ref{thm:ltbSIRT} -- Liouville-type result and stability]
	Let $(u,v)$ be the solution to~\eqref{e2.4}-\eqref{IDuv}.
	Since its initial datum $(0,0)$ is a sub-solution to~\eqref{e2.4}-\eqref{I0},
	the comparison principle implies that $(u,v)$ is non-decreasing in $t$.
	Observe that one can choose $M>v_*$ large enough so that the constant pair
		$M(\nu/\mu,1)$
		is a super-solution to~\eqref{e2.4}.
	It follows that $(u(t,x),v(t,x,y))$ converges as $t\to+\infty$, locally uniformly in space, to a bounded pair 
	$(u^r_\infty(x),v^r_\infty(x,y))$, and moreover, by the uniform regularity of solutions, 
	that $(u^r_\infty(t,x),v^r_\infty(x,y))$ is a (stationary) non-negative solution to~\eqref{e2.4}. 
	It remains to prove the Liouville-type result.
	We distinguish the two cases according to~$R_0$.

	\smallskip
	{\em Case $R_0>1$}.\\
		In this case $f$ fulfils all conditions in~\eqref{KPP} with ``$1$'' replaced by $v_*>0$.
		Hence, since $(u(t,x),v(t,x,y))$ is a super-solution to~\eqref{e1.2}, and 
		$v(t,x,y)>0$ for $t>0$, $x\in\RR$, $y>0$ due to  
			the parabolic strong maximum principle (because $I_0\not\equiv0$),
			 we infer from Proposition~\ref{t2.2} that
		$$
		(u^r_\infty(x),v^r_\infty(x,y))\geq\Big(\frac\nu\mu,1\Big)\,v_*,
		$$
the right-hand side being the unique positive solution to~\eqref{e1.2}.
	It is then straightforward to see,
	\blue{using the first equation in~\eqref{e1.2} alone},
	 that $v^r_\infty(x,y)\to v_*$ as $y\to+\infty$, uniformly in $x\in\RR$.
	Conversely, taking a diverging sequence 
	$(x_n)_{n\in\N}$ in $\RR$, the limit of the translations 
	$(u_\infty^r(x+x_n),v_\infty^r(x+x_n,y))$ (which exists locally in $(x,y)$ 
	by the uniform regularity of~$(u_\infty^r,v_\infty^r$))
	is a positive, stationary solution to~\eqref{e1.2}, hence by
	Proposition~\ref{t2.2} it coincides with $(\nu/\mu,1)v_*$.
	
	 It remains to prove the Liouville result. Let $(u_1,v_1)$ and $(u_2,v_2)$ be two pairs of
	positive, bounded, stationary solutions to~\eqref{e2.4}.
	Assume by way of contradiction that 
	$$k:=\max\bigg(\sup_{\RR}\frac{u_1}{u_2}\,,\,\sup_{\RR\times\RR^+}\frac{v_1}{v_2}\bigg)>1.$$
	Because of the limits we have just proved, one of the following situations necessarily occurs:
	$$\max_{\RR}\frac{u_1}{u_2}=k,\quad\text{or}\quad
	\max_{\RR\times\RR^+}\frac{v_1}{v_2}=k.$$
	Suppose we are in the latter case.
	Then, it is readily seen using the concavity of $f$ that the maximum 
	cannot be achieved in the interior of $\RR\times\RR^+$.
	Then, it is achieved at some point $(\bar x,0)$, and Hopf's lemma yields
	$$\partial_y(kv_2-v_1)(\bar x,0)>0.$$
	Using the second equation in~\eqref{e2.4}, together with $v_1(\bar x,0)=k v_2(\bar x,0)$, we find that
	$$ku_2(\bar x)=-\frac{kd}{\mu}\partial_y v_2(\bar x,0)+\frac{k\nu}{\mu}v_2(\bar x,0)
	<-\frac{d}{\mu}\partial_y v_1(\bar x,0)+\frac{\nu}{\mu}v_1(\bar x,0)=u_1(\bar x),$$
	which contradicts the definition of $k$.
	Consider the remaining case:
	$$\max_{\RR}\frac{u_1}{u_2}=k>\frac{v_1}{v_2}.$$
	Computing the difference of the equations satisfied by $ku_2$ and $u_1$ at a point 
	$\bar x$ where this maximum is achieved, we derive from~\eqref{MP} the contradiction
	$$0\leq D\mc{J} (ku_2-u_1)(\bar x)=
	\nu (v_1-kv_2)(\bar x,0)-\mu (u_1-ku_2)(\bar x)=\nu (v_1-kv_2)(\bar x,0)<0.$$ 
	We have thereby shown that $k\leq1$, that is, $(u_1,v_1)\leq(u_2,v_2)$.
	Exchanging the roles of the solutions, yields the uniqueness result.
	
	\smallskip
	{\em Case $R_0\leq1$}.\\
	We start with the Liouville-type result 	for non-negative solutions, with possibly $I_0\equiv0$.
	We need to show that, for any
	two pairs $(u_1,v_1)$, $(u_2,v_2)$ of non-negative, bounded, stationary solutions to~\eqref{e2.4},
	it holds that $(u_1,v_1)\leq(u_2,v_2)$. Assume by contradiction that, on the contrary, 
	$$h:=\max\Big(\frac\mu\nu\sup_{\RR}(u_1-u_2)\,,\,\sup_{\RR\times\RR^+}(v_1-v_2)\Big)>0.$$
	Suppose first that $\sup_{\RR\times\RR^+}(v_1-v_2)=h$, and let $((x_n,y_n))_{n\in\N}$
	be a maximising sequence. If $(y_n)_{n\in\N}$ is bounded from below away from $0$, then the functions 
	$v_j(x+x_n,y+y_n)$
	converge locally uniformly (up to subsequences) towards two 
	solutions $\t v_j$ of the equation
	$-d\Delta \t v_j=f( \t v_j)$ in a neighbourhood of the origin.
	Moreover, $(\t v_1-\t v_2)(0,0)=\max(\t v_1-\t v_2)=h$, and thus
	$$0\leq-d\Delta(\t v_1-\t v_2)(0,0)=f(\t v_1(0))-f(\t v_2(0))=f(\t v_2(0)+h)-f(\t v_2(0)).$$
	This is impossible because $f$ is decreasing on $\RR^+$.
	If, instead, $y_n\to0$ (up to subsequences), then the pairs
	$(u_j(x+x_n),v_j(x+x_n,y))$
	converge locally uniformly (up to subsequences) towards two 
	solutions $(\t u_j,\t v_j)$ of the same system, which is of the form~\eqref{e2.4} with $I_0$ either
	translated by some vector $(\xi,0)$, or replaced by $0$.
	Moreover the difference $\t v_1-\t v_2$ attains its maximal value~$h>0$ at~$(0,0)$ and satisfies
	$-d\Delta (\t v_1-\t v_2)=f( \t v_1)-f( \t v_2)$.
	As before, such a maximum cannot be attained also at some interior point, hence by Hopf's lemma 
	$$0>d\partial_y(\t v_1-\t v_2)(0,0)=
	\nu(\t v_1-\t v_2)(0,0)-\mu(\t u_1-\t u_2)(0)=h\nu-\mu(\t u_1-\t u_2)(0).$$
	Since $\sup(\t u_1-\t u_2)\leq \sup( u_1- u_2)$, we get a contradiction with the definition of $h$.
	
	Suppose now that
	\begin{equation}
	\label{h}
	h=\frac\mu\nu\sup_{\RR}(u_1-u_2)>\sup_{\RR\times\RR^+}(v_1-v_2).
	\end{equation}
	Consider a maximising sequence $(x_n)_{n\in\N}$ for $u_1-u_2$, then
	the limits (up to subsequences) $(\t u_j,\t v_j)$ of the translations 
	$(u_j(x+x_n),v_j(x+x_n,y))$, which, once again, satisfy a system analogous to~\eqref{e2.4}.
	The difference $\t u_1-\t u_2$ attains its maximum $\frac\nu\mu h$ at the origin, whence
	by~\eqref{MP}
		$$0\leq -D\mc{J} (\t u_1-\t u_2)(0)=
	\nu (\t v_1-\t v_2)(0,0)-\mu (\t u_1-\t u_2)(0)=\nu (\t v_1-\t v_2)(0,0)-\nu h.$$ 
	This contradicts~\eqref{h}	
	The proof of the Liouville result  is concluded. 
	
	Let us pass to the limits at infinity.
	The limit $v_\infty^r(x,y)\to0$ as $y\to+\infty$, uniformly with respect to $x$,
	readily follows from the negativity of $f$ on~$(0,+\infty)$.
	Consider now a diverging sequence $(x_n)_{n\in\N}$ in $\RR$.
	The sequence of translations 
	$(u_\infty^r(x+x_n),v_\infty^r(x+x_n,y))$ converges locally uniformly (up to subsequences)
	towards a bounded, stationary solution $(\t u,\t v)$ to~\eqref{e2.4} with $I_0\equiv0$.
	We apply the Liouville-type we have just proved and infer that $(\t u,\t v)\equiv(0,0)$.
	This concludes the proof of the theorem.	
  \end{proof}

We conclude the proof of Theorem~\ref{thm:ltbSIRT} using the plane wave solutions of Section~\ref{spreadingspeed}.
\begin{proof}[Proof of Theorem~\ref{thm:ltbSIRT} -- exponential decay] 
Let us start with the case $R_0<1$.
In the first place, we linearise system~\eqref{e2.4} around $v=0$ and we get~\eqref{e4.1}, where $f'(0)=\alpha(R_0-1)<0$.
We look for steady waves of the form~\eqref{expsol} with $c=0$.
With analogous computation as in Section~\ref{spreadingspeed}, we end up with the system
\begin{equation}
\label{e3.3}
\begin{cases}
\displaystyle b=\frac\nu d\left(\frac{\mu}{\mu-D\vp_L(a)}-1\right)\\
a^2+b^2 =\di-\frac1df'(0)\\
\gamma=\di\frac\mu{\nu+db}\;.
\end{cases}
\end{equation}
%
Since we need $\gamma>0$, the last equation yields $b>-\nu/d$, which in turns implies, owing to the first equation, 
$|a|\!<\!a^\infty(D,L)$, where
$a^\infty(D,L)$ is the only positive root of ${D\vp_L(a)\!=\!\mu}$.
In the $(a,b)$ plane, the first equation is the graph of a convex, even function $b(a)$ vanishing at $0$ and with
asymptotes $a=\pm a^\infty(D,L)$, while the
second one is a circle centred at the origin with radius~$\sqrt{-f'(0)/d}=\sqrt{\frac{\alpha}d(1-R_0)}$. 
Hence there are two solutions~$(\pm a_*,b_*,\gamma_*)$,
with 
\begin{equation}
\label{a*<}
0<a_*<\min\Big\{a^\infty(D,L),\sqrt{\frac{\alpha}d(1-R_0)}\Big\},\qquad
b_*,\gamma_*>0.
\end{equation}
Let us call $(u^\pm,v^\pm)$ the corresponding steady waves.
Take then $\overline k>0$ large enough so that $\overline kv^\pm>v^r_\infty$ in the support of $I_0$.
We finally define
$$\overline u:=\min(u^r_\infty,\overline ku^-,\overline ku^+),\qquad \overline v:=\min(v^r_\infty,\overline kv^-,\overline kv^+).$$ 
Since $\overline v=v^r_\infty$ whenever $I_0>0$, the pair $(\overline u,\overline v)$ is a 
generalised super-solution to~\eqref{e2.4}.
By comparison, the solution $(u,v)$ to~\eqref{e2.4}-\eqref{IDuv} stays below 
$(\overline u,\overline v)$ for all times, and we then deduce from the stability result of
Theorem~\ref{thm:ltbSIRT}
that $(u_\infty^r,v_\infty^r)\leq(\overline u,\overline v)$, i.e.
$$(u_\infty^r(x),v_\infty^r(x,y))\leq \overline ke^{-a_*|x|}(1,\gamma_* e^{-b_*  y}).$$

Let us turn to the lower bound. Steady waves with $b=b_*$, $\gamma=\gamma_*$ and any~$a>a_*$,
are sub-solutions to the linearised problem~\eqref{e4.1}. 
However, in order to get suitable sub-solutions for the nonlinear problem, we penalise $f'(0)$
by a small $\delta>0$ and we also truncate
the domain at a large value $Y$ of~$y$, that is, we consider~\eqref{penalised}.
As shown in the Appendix, this translates into a slight perturbation of the original system~\eqref{e4.1},
namely, for given $a>a_*$, the truncated wave
\begin{equation}
	(\underline u(x),\underline v(x,y))=e^{-ax}\big(1,\gamma_*(e^{-b_*  y}-e^{b_*y-2b_*Y})\big),
\end{equation}
is a sub-solution to the penalised system provided $\delta$ is sufficiently small and $Y$ is large. 
We take $\underline k>0$ small enough so that $\underline k(\underline u,\underline v)$ is a sub-solution to~\eqref{e2.4}
in the half-strip $x>0$, $0\leq y\leq Y$
and, in addition,
$\underline k(\underline u(0),\underline v(0,y))<(u^r_\infty(0),v^r_\infty(0,y))$ for $0\leq y\leq Y$.
Next, take $M$ large enough so that $M(\nu/\mu,1)$ is a super-solution to~\eqref{e2.4} and moreover
$M(\nu/\mu,1)>\underline k(\underline u,\underline v)$ for $x\geq0$, $0\leq y\leq Y$.
The solution $(\tilde u,\tilde v)$ with initial datum $M(\nu/\mu,1)$ is non-increasing in $t$ and 
converges to a non-negative steady state,
which is necessarily $(u_\infty^r,v_\infty^r)$ owing to the Liouville result.
In particular $(\tilde u(t,0),\tilde v(t,0,y))>\underline k(\underline u(0),\underline v(0,y))$ for $t>0$, 
$0\leq y\leq Y$. We can then apply the comparison principle in the half-strip $x>0$, $0\leq y\leq Y$,
and infer that $(\tilde u,\tilde v)>\underline k(\underline u,\underline v)$ there, for all $t>0$.
Passing to the limit $t\to+\infty$ yields $(u_\infty^r,v_\infty^r)\geq \underline k(\underline u,\underline v)$
for $x\geq0$. The specular estimate for $x<0$ holds true by the symmetry of $(u_\infty^r,v_\infty^r)$.
The proof of the case $R_0<1$ is thereby concluded owing to the arbitrariness of $a>a_*$.

In the case $R_0>1$ the argument is analogous.
One linearises the system around $(\nu/\mu,1)v_*$ and gets~\eqref{e3.3} with $f'(0)$ replaced by $f'(v_*)$,
which is also negative. One then finds the super-solutions to~\eqref{e2.4} outside the 
support of $I_0$ in the form $(\nu/\mu,1)v_*+(u^\pm,v^\pm)$,
the sub-solution in the form $(\nu/\mu,1)v_*+\underline k(\underline u,\underline v)$, and concludes as before.
\end{proof}

Let us study the exponential rate $a_*$ in  Theorem~\ref{thm:ltbSIRT} 
for extreme values of $L$, that is, 
when the range of contaminations on the road is large or small.
Let us focus, as in the above proof, on the case $R_0<1$.
Recall that $a_*<a^\infty(D,L)$ by~\eqref{a*<}, where $a^\infty(D,L)$ is defined by
$D\vp_L(a^\infty(D,L))=\mu$.
  Since $\vp_L(a)=\vp_1(aL)$, one derives 
 $a^\infty(D,L)\to0$ as~$L\to+\infty$, hence the same is true for~$a_*$, and then, form the equation of the circle
 in~\eqref{e3.3}, 
$$
b_*\to\sqrt{\frac{\alpha}d(1-R_0)}=:\rho
\quad\text{ as }\;L\to+\infty.
$$
Thus, from the first equation in~\eqref{e3.3} one gets
\begin{equation}
	\label{L>>}
D\vp_L(a_*)=D\vp_1(a_*L)\to\frac{d\mu \rho}{\nu+d\rho}
\quad\text{ as }\;L\to+\infty.
\end{equation}
If, on the contrary, $L$ is small --~that is, we are close to the classical local diffusion~-- we have
\begin{equation}
	\label{L<<}
D\vp_L(a)=D\vp_{1}(aL)\sim Da^2L^2\langle x^2K\rangle\to0\quad\text{ as }\;L\to0,
\end{equation}
locally uniformly in $a$.
This yields, thanks to the 
first equation in~\eqref{e3.3}, that $b\to0$ as $L\to0$, whence, by the second equation,
that $a_*\to\rho=\sqrt{\frac{\alpha}d(1-R_0)}$ as $L\to0$. This limit indeed coincides with the asymptotic exponential rate 
that holds for the local model, see \cite[Theorem 4.2]{SIRT1}.
In any case, since by \eqref{phi''} we have that $\vp_L(a)=\vp_1(aL)\geq \langle x^2K\rangle a^2L^2$,
it follows that 
\begin{equation}
\label{a*<<DL2}
a_*<a^\infty(D,L)\leq \sqrt{\frac{\mu}{\langle x^2K\rangle DL^2}}.
\end{equation}
This shows that $a_*$ can be small, i.e.~the solution has a thick tail, even if $D$ is small, 
provided that $L$ is large, or, on the contrary, if $L$ is small but $D$ is sufficiently large.

We conclude with the proof of Theorem~\ref{thm:cSIRT}, which just consists in 
showing that the compactly supported function $I_0$ does not affect the behaviour of 
the solution far from the origin.

\begin{proof}[Proof of Theorem~\ref{thm:cSIRT}]
Since $R_0>1$, the function $f$ is of the Fisher-KPP type. 
We call $\cSIRT$ the speed $c_*$ provided by Theorem \ref{t2.1} with such a $f$.
Hence the comparison between $\cSIRT$ and the standard speed $\cSIR$ stated in Theorem~\ref{thm:cSIRT} hold,
with $D_*$ given by~\eqref{D*}. It remains to show that $\cSIRT$ is actually the spreading speed 
for~\eqref{e2.4}-\eqref{I0}.

The pair $(u,v)$ is a super-solution to~\eqref{e1.2}, hence by the comparison principle
and Theorem \ref{t2.1} we infer
		$$
		\forall \e<\cSIRT,\quad
		\liminf_{t\to+\infty}\inf_{\vert x\vert\leq(\cSIRT-\e)t}\big(u(t,x),v(t,x,y)\big)
				\geq\Big(\frac\nu\mu,1\Big)v_*,
		$$
locally uniformly with respect to $y\geq0$ (in formula~\eqref{spreading} one has $v_*=1$ as the positive zero of $f$). 
Recall from Theorem~\ref{thm:ltbSIRT} that~$(\nu/\mu,1)v_*$ is the limit as $x\to\pm\infty$
of the steady state $(u^r_\infty,v^r_\infty)$.
We further know from Theorem~\ref{thm:ltbSIRT} that $(u,v)\to(u_\infty^r,v_\infty^r)$
		as $t\to+\infty$ locally uniformly in space, and, by comparison, that 
$(u,v)\leq(u^r_\infty,v^r_\infty)$.
All this facts together imply the validity of the first limit of Theorem~\ref{thm:cSIRT}.

The second limit directly follows by comparison with the plane waves provided by the case 2 of Lemma~\ref{lem:super}
(recall that $c_*$ in Theorem \ref{t2.1} is precisely given by $c_*(D,L)$
of Lemma~\ref{lem:super}). Indeed, being super-solutions to the linearised system~\eqref{e4.1},
it is clear that they are super-solutions to~\eqref{e2.4} as well, up to multiplication by 
a large constant.
\end{proof}


\subsection{The case of pure transport on the road}

Another way of analysing a nonlocal effect of a road on the spreading of an epidemic, 
is by considering a pure transport equation on the line.
Namely, we introduce the~system
\begin{equation}
\label{SIRTq}
\left\{
\begin{array}{rll}
\partial_tI-d\Delta I+\alpha I\,=&\beta SI &\quad(t>0,\ (x,y)\in\RR^2_+)\\
\partial_tS\,=&-\beta SI &\quad(t>0,\ (x,y)\in\RR^2_+)\\
-d\partial_yI\,=&\mu T-\nu I &\quad(t>0,\ x\in\RR,\ y=0)\\
\partial_t T+q\partial_x T\,=&\nu I-\mu T &\quad(t>0,\ x\in\RR,\ y=0).
\end{array}
\right.
\end{equation}
where $q\in\RR$ is a given constant. The system for the cumulative densities $u,v$ reads
\begin{equation}
\label{Cauchyq}
\left\{
\begin{array}{rll}
\partial_tv-d\Delta v\,=&f(v)+I_0(x,y) &\quad(t>0,\ (x,y)\in\RR^2_+)\\
-d\partial_yv\,=&\mu u-\nu v &\quad(t>0,\ x\in\RR,\ y=0)\\
\partial_t u+q\partial_x u\,=&\nu v-\mu u &\quad(t>0,\ x\in\RR,\ y=0),
\end{array}
\right.
\end{equation}
with $f$ given by~\eqref{f(v)}.
As in Section \ref{spreadingspeed},
the spreading speed $c_*$ for this new system will be given by the minimal $c$ such that the 
linearised system admits plane wave super-solutions of the type~\eqref{expsol}.
As a matter of fact, because the system is no longer symmetric in the $x$ variable, there will be
two distinct spreading speeds, $c_*^\pm$, one leftward and one rightward. 

\begin{theorem}\label{spreading-q}
 Assume that $R_0>1$. Let $(u,v)$ be the solution to~\eqref{Cauchyq}, \eqref{f(v)}-\eqref{IDuv}. 
	Then, there exist $c_*^\pm>0$ such that, for all $\e>0$, it holds
$$
\liminf_{t\to+\infty}\inf_{-(c_*^--\e)t\leq x\leq(c_*^+-\e)t}\big|(u(t,x),v(t,x,y))\big|>0,
$$
$$
		\lim_{t\to+\infty}\sup_{x\leq-(c_*^-+\e)t}\big|(u(t,x),v(t,x,y))\big|=0,\qquad
		\lim_{t\to+\infty}\sup_{x\geq (c_*^++\e)t}\big|(u(t,x),v(t,x,y))\big|=0,
$$
locally uniformly with respect to $y\geq0$.

In addition, the spreading speeds $c_*^\pm$ satisfy
		$$c_*^\pm\,\begin{cases} = \cSIR & \text{if }\pm q\leq \cSIR\\
							\in(\cSIR,|q|) & \text{if }\pm q>\cSIR
				\end{cases}
				\qquad\text{with }\;\cSIR:=2\sqrt{d\alpha(R_0-1)}.$$	

Finally, $c_*^\pm/|q|$ converge to a positive constant $\kappa_*\in(0,1)$ as $q\to\pm\infty$.
\end{theorem}

Recalling that $S=S_0 e^{-\beta v}$, the above result shows that the epidemic wave moves at the leftward and rightward
asymptotic speeds $c_*^\pm$ respectively.
These speeds are always no less than $\cSIR$, which means that the transport term $q$ on the line
does not slow down the spreading speed in the opposite direction, 
no matter how strong it is.
On the contrary, as soon as the intensity $q$
is larger than the classical speed $\cSIR$, the spreading speed in the direction of the transport
is enhanced, but, however, it never reaches the value of $q$ itself.

\begin{proof}[Proof of Theorem~\ref{spreading-q}]
The problem of the existence of plane wave solutions
for the linearised system reduces to the algebraic system
\begin{equation}
\label{algebraicq}
\left\{
\begin{array}{rcl}
(c-q)a&=&-\di\frac{d\mu b }{\nu+db }\phantom{\frac{b^2}{b^2_{\cK}}}\\
-(a^2+b^2)+\di\frac{ca}d&=&\di\frac{\cK^2}{4d^2}\,.
\end{array}
\right.
\end{equation}
Consider the case $q\leq \cK$.
Recall from Section~\ref{spreadingspeed} 
that the set of solutions $\Gamma_2$ of the second equation in~\eqref{algebraicq} 
is non-empty if and only if
$c\geq \cK$, and that it is a circle. 
For $c=\cK$ the circle reduces to the point $(a,b)=(\cK/(2d),0)$, which satisfies the 
inequality ``$\geq$'' in the first equation  of~\eqref{algebraicq}.
It follows that in such a case,~\eqref{Cauchyq} admits a super-solution of the type
\eqref{expsol} if and only if $c\geq c_*(q):= \cK$.

Next, consider the case $q>\cK$.
For $c\geq q$, the set $\Gamma_2$ is non-empty and
any point $(a,b)\in \Gamma_1$ with $a,b>0$ satisfies the inequality ``$>$'' in 
the first equation of~\eqref{algebraicq},
that is,~\eqref{Cauchyq} admits super-solutions of the form
\eqref{expsol}. Conversely, for $c=\cK$, $\Gamma_2$ reduces to the point $(\cK/(2d),0)$, which 
satisfies ``$<$'' in 
the first equation of~\eqref{algebraicq}, i.e., the corresponding plane wave is a sub-solution to the linearised system.
This means that there is a first value $c=c_*(q)\in(\cK,q)$ at which the two curves in~\eqref{algebraicq}
are tangent,
and which provides us with a plane wave super-solution to~\eqref{Cauchyq}.

Let us investigate the behaviour of $c_*$ as $q\to+\infty$.
We write $c=\kappa q$ with $\kappa>0$ and $\alpha= aq$. The system rewrites as
\begin{equation}
\label{algebraicqBIS}
\left\{
\begin{array}{rcl}
(1-\kappa)\alpha&=&\di\frac{d\mu b }{\nu+db }\phantom{\frac{b^2}{b^2_{\cK}}}\\
-\Big(\frac{\alpha^2}{q^2}+b^2\Big)+\di\frac{\kappa \alpha}d&=&\di\frac{\cK^2}{4d^2}\,.
\end{array}
\right.
\end{equation}
Dropping the term $\alpha^2/q^2$ (that will be justified at the end of the computation)
we get
$$\Big(\frac{\cK^2}{4d^2}+b^2\Big)\frac{1-\kappa}{\kappa}=\frac{d\mu b }{\nu+db }.$$
Consider $k\in(0,1)$. For $\kappa\sim 0$ this equation does not admit solution $b\geq0$, whereas 
 for $\kappa\sim 1$ it does.
 There exists then a minimal $\kappa=\kappa_*\in(0,1)$ for which a solution exists.
 We then recover the same existence result for system~\eqref{algebraicqBIS} when $q\to+\infty$,
 with a minimal value $\kappa_q\to\kappa _*$.
 This shows that $c_*\to\kappa_*q$ as $q\to+\infty$.  
\end{proof}



\section{Discussion}\label{s7}

 We have proposed and analysed a model that quantifies the effect of a line of fast, nonlocal diffusion, both on the propagation of fronts for models of reaction-diffusion for biological invasions,
and of the $SIR$ type for  epidemics. Such a modelling for the diffusion processes on the line is relevant, as it makes possible  long range displacements along lines of communications in addition to local ones.  This type of effects is widely recognised to exist. Our aim here is to provide a rigorous formulation of this feature and to analyse it mathematically.
 While we had previously analysed the effect of a line having a diffusion of its own on reaction-diffusion propagation \cite{BRR2},  \cite{SIRT1}, we had  considered there a diffusion given by the standard Laplacian. The nonlocal dispersal that we are proposing here encompasses and confirms the results that we have already obtained, and introduces further elements of understanding.  

%
The nonlocal diffusion on the line is characterised by two parameters: its intensity~$D$, that essentially measures the importance of the traffic, and the parameter $L$, that   represents the characteristic length of individual travel. 
 This additional parameter amplifies the effect of the road on all the aspects of the overall dynamics.

Regardless of the size of all other parameters, we have shown that the spreading velocity along the line behaves like 
$\sqrt{DL^2}$. This really shows that the dynamics on the line is that of an effective reaction-diffusion process of the 
Fisher-KPP type, with the kernel $DK_L(x)$, and a reaction term of the form $\gamma u(1-u)$, 
with $\gamma$ tailored so as to obtain the spreading speed. This broadens the picture that we had already 
obtained when the diffusion on the line is given by a standard diffusion operator 
$-D\partial_{xx}$. In \cite{BRR2}, 
we showed that  the spreading velocity grows like $\sqrt D$. 
Such an effect was also observed on long range diffusion operators of the form 
$(-\partial_{xx})^\alpha$  ($0<\alpha<1$) in \cite{BCRR}. 
In this last case we proved that the spreading speed is exponential. 
The first noticeable effect of the nonlocal diffusion that we are considering is that 
the spreading velocity can be quite large if the range $L$ of the dispersal on the road is large, even if the intensity of the traffic is modest.


 Let us concentrate on the model for the propagation of epidemics. The pandemic threshold $R_0$ is the same through all the models we have studied so far, from the classical $SIR$ model to the nonlocal model we are dealing with in this paper. However, the presence of the nonlocal operator on the line has an important
 quantitative impact on the system.
 Let us first focus on the case $R_0<1$. The epidemic does not spread, 
 however the population reaches a final state 
 where much more infected individuals may nonetheless be found at very large distance from the origin of the outbreak
 in the presence of the road than without it. This is reflected in the asymptotics of the limit, 
 as $t\to+\infty$, of the cumulative density $I_{tot}$ of infected, that is $S_0(1-e^{-\beta v_\infty^r})$, 
 which decays exponentially at infinity with decay rate $a_*$, (compare~Theorem~\ref{thm:ltbSIRT}).
 The bound~\eqref{a*<<DL2} indeed shows that, even if~$D$ is small -- which would, in principle, make the decay exponent $a_*$ 
 quite large, thus granting a fast exponential decay -- the parameter $L$ can be made large enough to make
 $a_*$ arbitrarily small.
 If, on the contrary, $L$ is small, thus once again, potentially allowing $a_*$ to be quite large,
 this effect can be compensated by a very large diffusion coefficient $D$. According to~\eqref{a*<<DL2},
 what really matters is indeed the size of $DL^2$, which, if large, yields a small decay rate $a_*$.  
 
 
 The most important effects can be observed on the propagation velocity~$\cSIRT$. 
 Let~us focus on the case when $R_0$ is only slightly larger than 1, that is, a range where the epidemics progress would be expected to be slow and thus would  not appear to pose a major public health concern. In such a case however
 the propagation speed is accelerated by a factor of order $\sqrt{DL^2}$, see~\eqref{e.param.12} (or also~\eqref{w*}).  
 Thus,  the size of $L$ gives an additional important boost to enhancement. In particular, few individuals moving very far are sufficient to produce an important increase of the spreading speed of the epidemics, all other parameters being small. We retrieve, in a way that is even stronger than in \cite{SIRT1}, 
 the propagation enhancement even in the case of a seemingly mild epidemic wave.  
 
 We also observe that the factor $w_*$ appearing in the asymptotics~\eqref{w*} for the spreading speed  
  is {\em decreasing} with respect to~$d$ (see the comment at the end of Section~\ref{spreadingspeed}). 
 This  monotonicity is rather counterintuitive, and is yet another manifestation of the complexity of the interaction
 between the dynamics in the field and that on the road. One possible interpretation is the following: the flux of individuals from the field onto the road, $dv_y$, is proportional to $d$. It is also equal to $\mu u-\nu v$, which
 is a negative quantity, at least for the linear waves. As a consequence, it is all the more negative as $d$ is large, 
 thus weakening the effect of the road and resulting in a slowdown effect when $d$ becomes larger.
 One could show that this type of monotonicity (that we observe here for the first time) also holds for local diffusion models.

On the mathematical side, one observes an unexpected preservation of smoothness 
of solutions due to the interaction between the line and the upper half plane, that is not present in the classical nonlocal Fisher-KPP models.

We have finally discussed the effect of a pure unidirectional transport on the line, and we have found another surprising result. The transport on the line does enhance the overall propagation, but with an important subtlety. If $q$ is the 
velocity of the transport,  as $q\to+\infty$, the spreading speed in that direction tends to~$+\infty$
as $\kappa_*q$, with $\kappa_*$ positive but {\em strictly smaller} than $1$.

The fact that the spreading speed is strictly smaller than $q$ (and that $\kappa_*<1$)
can be interpreted as follows: infected individuals $T$
are transported by the road with a speed~$q$, but if the latter is larger than
the speed $\cSIR$ in the field, the incoming/outcoming contribution 
of infected at their location, i.e.~$\nu I-\mu T$,  is negative,
and this slows down the speed of propagation. 

This would not have been the case if 
contamination took place on the line too.
Indeed, we showed in~\cite{BRR3} that, for the model of biological invasion
with local diffusion on the road, the spreading speed behaves like~$q$ (with factor $1$)
as $q\to+\infty$ 
provided that a reaction term is also present on the road, and it is sufficiently large
compared with the exchange rate $\mu$, see \cite[Theorem~1.3]{BRR3}.
As a corroboration of the above interpretation, 
one can indeed check in our proof  that, in the present case, 
the spreading speed tends to $q$ as $\mu\to0$, and the limit factor $\kappa_*$
tends to $1$.


\section*{Appendix}

%
%
%

 In this appendix, we construct a solution of the form 
 $(\tilde u_c(x-ct),\tilde v_c(x-ct,y))$, with  $c<c_*(D,L)$ close enough to $c_*(D,L)$,
 to the penalised problem~\eqref{penalised}, 
 with $\delta>0$ sufficiently small and $Y>0$ sufficiently large.
%
 In order to fulfil the Dirichlet condition at $y=Y$, we consider a variant of the plane waves~\eqref{expsol}, namely
 \begin{equation}
 \label{expsolperturbed}
 (u(t,x),v(t,x,y))=e^{-a(x-ct)}\big(1,\gamma(e^{-b  y}-e^{by-2bY})\big).
 \end{equation}
 Plugging it into~\eqref{penalised} yields
 $$\gamma=\frac{\mu}{\nu(1-e^{-2bY})+db(1+e^{-2bY})},
\qquad
b=g^{-1}_Y\circ G_1^c(a)=G^c_{2,\delta}(a),$$
 where $G_1^c$ is defined by~\eqref{G1} and
 $$g_Y(b):=b\coth(Yb),\qquad
 G^c_{2,\delta}(a):=\frac1{2d}\sqrt{c^2-\cK^2+4d\delta-(2d a-1)^2}.$$
 One sees that $g_Y(b)$ is increasing for $b\geq0$ and converges locally to the identity, together with its derivatives,
 as $Y\to+\infty$.
 When $\delta\to0$ and $Y\to+\infty$ we end up with the previous system~\eqref{e4.3}, that we recall admits
 solution if and only if $c\geq c_*(D,L)$; for $c=c_*(D,L)$ the unique solution is
 $a_*,b_*,\gamma_*>0$, with
 $$b_*=G_1^{c_*(D,L)}(a_*)=G^{c_*(D,L)}_{2,0}(a_*).$$
 One also directly checks that $(G_1^c)''(a)>0$, at least for the values $a>0$ where~$G_1^c(a)\geq0$, and that 
 $(G^{c_*(D,L)}_{2,0})''<0$ in its domain.
 Let us call $h^c(a):=G_1^c(a)-G^c_{2,0}(a)$. This is an analytic function that satisfies 
 $h^c(a_*)\searrow0$ as $c\nearrow c_*(D,L)$
 (with strict monotonicity) and moreover $(h^{c_*(D,L)})''(a_*)>0$.
 These properties allow us to apply Rouché's theorem and find, for $c<c_*(D,L)$ close enough to $c_*(D,L)$,
 two complex solutions of $h^c(a)=0$ with nonzero imaginary part of order $\sqrt{h^c(a_*)}$
 (see \cite[Appendix B]{BRR2} for the detailed argument).
 The same properties are fulfilled also by the perturbation $h_{\delta,Y}^c:=g^{-1}_Y\circ G_1^c(a)-G^c_{2,\delta}(a)$
 of~$h^c(a)$ (with $c_*(D,L)$, $a_*$ replaced by some slightly different values) 
 provided $\delta$ is sufficiently small and $Y$ is sufficiently large.
 As a consequence, for $c<c_*(D,L)$ close enough to $c_*(D,L)$, we can find $\delta$ small and $Y$ large so that
 the equation $h_{\delta,Y}^c(a)=0$ admits two complex solutions with nonzero imaginary part of order $\sqrt{h^c(a_*)}$.
 Pick one of them, together with the associated $b=G^c_{2,\delta}(a)$ and $\gamma$. This provides us with 
 a complex plane-wave solution $( u(x-ct), v(x-ct,y))$ to~\eqref{penalised}.
Finally, its real part
$$\big(\tilde u_c(x-ct),\tilde v_c(x-ct,y)\big):=\big( \Re(u)(x-ct),\Re(v)(x-ct,y)\big),$$
is a real solution to~\eqref{penalised}. The positivity sets of $\tilde u_c$, $\tilde v_c$
fulfil~\eqref{u>0},~\eqref{v>0} with $\omega_c=1/\Im(a)\to+\infty$ as $c\nearrow c_*(D,L)$.



\end{document}